\documentclass[11pt,a4paper]{article}

\usepackage[utf8]{inputenc}
\usepackage{fullpage}
\usepackage{authblk}
\usepackage{blindtext}
\usepackage{blindtext}
\usepackage{boxedminipage}
\usepackage{amsfonts}
\usepackage{amsmath} 
\usepackage{amssymb}
\usepackage{amsthm}
\usepackage{verbatim}
\usepackage{t1enc}
\usepackage{subfig}
\usepackage{dsfont}
\usepackage{mathtools}
\usepackage{mathrsfs}
\usepackage{algorithm}
\usepackage{enumitem} 
\usepackage{algpseudocode}
\usepackage{hyphenat}
\usepackage{import}
\usepackage{color}   
\usepackage[dvipsnames]{xcolor}
\usepackage{hyperref}
\hypersetup{
    colorlinks=true, 
    linkcolor= blue, 
    citecolor= ForestGreen,
    urlcolor=red, 
    linktoc=all 
}
\usepackage{cleveref}
\crefname{enumi}{part}{parts}
\crefname{equation}{}{}

\usepackage{fontenc}
\usepackage{graphicx}
\graphicspath{ {images/} }
\usepackage{subfig}

\newtheorem{theorem}{Theorem}[section]
\newtheorem{lemma}[theorem]{Lemma}
\newtheorem{proposition}[theorem]{Proposition}
\newtheorem{corollary}[theorem]{Corollary}

\theoremstyle{definition}
\newtheorem{definition}[theorem]{Definition}

\newtheorem{hypothesis}[theorem]{Hypothesis}
\newtheorem{problem}[theorem]{Problem}

\theoremstyle{remark}

\newcommand{\Sym}[0]{\mathrm{Sym}}

\newcommand{\lists}[3]{ {#1}_1 {#2} {#1}_2 {#2} \ldots {#2} {#1}_{#3}  }

\newcommand{\factor}{R}
\newcommand{\InP}{\mathfrak{InP}}
\newcommand{\MAut}{\mathrm{MAut}}
\newcommand{\Sy}{\mathrm{S}}

\newcommand{\Cy}{\mathrm{C}}
\newcommand{\Di}{\mathrm{D}}
\newcommand{\Ker}{\mathrm{Ker}}
\newcommand{\Supp}{\mathrm{Supp}}

\algnewcommand{\IIf}[1]{\State\algorithmicif\ #1\ \algorithmicthen}
\algnewcommand{\EndIIf}{\unskip\ \algorithmicend\ \algorithmicif}
\algnewcommand{\Forr}[1]{\State\algorithmicfor\ #1\ \algorithmicdo}
\algnewcommand{\EndForr}{\unskip\ \algorithmicend\ \algorithmicfor}

\title{Computing normalisers of intransitive groups}
\date{\today}

\author[a]{Mun See Chang}
\author[a]{Christopher Jefferson}
\author[b]{Colva M.\ Roney-Dougal}
\affil[a]{{\footnotesize School of Computer Science, University of St Andrews, St Andrews, KY16 9SX, United Kingdom.}}
\affil[b]{{\footnotesize School of Mathematics and Statistics, University of St Andrews, St Andrews, KY16 9SS, United Kingdom.}}
\affil[ ]{{\footnotesize\texttt {\{msc2, caj21, colva.roney-dougal\}@st-andrews.ac.uk}}}

\begin{document}

\maketitle

\begin{abstract}
The normaliser problem takes as input subgroups $G$ and $H$ of the symmetric group $\Sy_n$, and asks one to compute $N_G(H)$. 
The fastest known algorithm for this problem is simply exponential, whilst more efficient algorithms are known for restricted classes of groups. In this paper, we will focus on groups with many orbits. We give a new algorithm for the normaliser problem for these groups that performs many orders of magnitude faster than previous implementations in \texttt{GAP}. 
We also prove that the normaliser problem for the special case $G=\Sy_n$ is at least as hard as computing the group of monomial automorphisms of a linear code over any field of fixed prime order. 
\end{abstract}


\section*{Keywords}
Computational group theory; Permutation groups; Backtrack search.

\section{Introduction}
\label{section:intro}


Given generators for subgroups $G$ and $H$ of the symmetric group $\Sy_n$, the \emph{normaliser problem} asks one to compute a generating set for the normaliser $N_{G}(H)$. 
The fastest known bound for this problem, in general, is Wiebking's simply exponential bound $2^{O(n)}$ \cite{wiebking}.
Better bounds are known for various cases: for example, quasipolynomial $2^{O(\log^3{n})}$ if $N_{\Sy_n}(H)$ is primitive \cite{normPrim, primNorm}, and polynomial if $G$ has restricted composition factors by work of Luks and Miyazaki \cite{luksMiyazaki}. 
In terms of practical computation, there are many algorithms to solve special cases of the normaliser problem (see \cite{holtNorm, hulpkeNormUsingAuts, miyamoto} for examples).

As a consequence of Babai's quasipolynomial solution to the string isomorphism problem~\cite{babaiGI}, the intersection of permutation groups can be computed in quasipolynomial time. 
So, with a quasipolynomial cost, it suffices to compute $N_{\Sy_n}(H)$, then $N_{G}(H) = N_{\Sy_n}(H) \cap G$. 
Furthermore, in practice, computing intersections is much faster than computing normalisers.  
We shall therefore focus on the following problem.
\begin{problem}[\textsc{Norm-Sym}]
Given $H = \langle X \rangle \leq \Sy_n$, compute $N_{\Sy_n}(H)$.
\end{problem}

To better understand its worst-case complexity, it is helpful to study the case where the problem seems to be the hardest. 
In this paper, we will consider intransitive groups, for which \textsc{Norm-Sym} is not known to be solvable in quasipolynomial time. 
Since $N_{\Sy_n}(H)$ may only permute permutation isomorphic $H$-orbits, the following case seems likely to be hardest. 

\begin{definition}\label{our H}
Let $A \leq \Sy_m$ be transitive. 
Let $\mathfrak{InP}(A)$ consist of all groups $H \leq \Sy_{mk}$, for any $k$, such that $H$ has orbits $\mathcal{O} = \{ \Omega_1, \Omega_2, \ldots, \Omega_k \}$ and each restriction $H|_{\Omega_i}$ is permutation isomorphic to $A$.
\end{definition} 

In particular, we will consider $\InP(A)$ with $A$ the regular representation of the cyclic group $\Cy_p$ or the natural representation of the dihedral group $\Di_{2q}$ of order $2q$, where $p$ and $q$ are primes and $q$ is odd.

Luks proved that the graph isomorphism problem is polynomial-time reducible to \textsc{Norm-Sym} \cite{luksHierarchy}, which is a special case of the normaliser problem. 
However, not much more is known about where \textsc{Norm-Sym} fits into the complexity hierarchy. In particular, we do not know its relation to the intersection problem, or whether it is polynomial-time equivalent to the general normaliser problem. 
In 1993, Luks asked \cite{luksHierarchy} for a decision problem polynomial-time equivalent to \textsc{Norm-Sym}, but this remains open. 
Notably, we do not know if \textsc{Conj-Sym}, the problem of deciding if subgroups $H_1$ and $H_2$ of $\Sy_n$ are conjugate in $\Sy_n$, is polynomial-time equivalent to \textsc{Norm-Sym}.

Let \textsc{MAut} be the problem of computing the monomial automorphism group of a linear code $C$, and let \textsc{MEq} be the problem of deciding if linear codes $C_1$ and $C_2$ are monomially equivalent (see \Cref{defn: monomial aut} for details).  
Petrank and Roth showed that \textsc{MEq} for codes of length $k$ over $\mathds{F}_p$ is at least as hard as the graph isomorphism problem \cite{codeEquivGI}, and \textsc{MEq} has time complexity bounded by $2^{O((p-1)k)}$, assuming constant time field operations \cite{sendrierHardnessCodeEquiv, BabaicodeEquivandGI}. 
Our first main complexity result is the following. 

\begin{theorem} \label{complexity main theorem}\label{norm of our groups polytime equiv to code aut}
For a fixed prime $p$, \textsc{Norm-Sym} and \textsc{Conj-Sym} for groups in $\InP(\Cy_p)$ are polynomial-time equivalent to \textsc{MAut} and \textsc{MEq}, respectively, for codes over $\mathds{F}_p$. 
\end{theorem}

As a consequence, \textsc{Norm-Sym} is at least as hard as \textsc{MAut} for codes over any fixed field of prime order (see \Cref{cor: rel complexity}).

As Wiebking's simply exponential algorithm \cite{wiebking} is not implemented, 
the fastest implemented algorithm to compute $N_{\Sy_n}(H)$ has a runtime complexity of $2^{O(n\log{n})}$.
Using methods based on the work of Feulner \cite{feulner}, we shall bound the complexity of \textsc{Norm-Sym} for $H \in \mathfrak{InP}(\Cy_p)$ by $2^{O(\frac{n}{p}\log{\frac{n}{p}} + \log{n})}$.
We show that we may also bound the depth of the search tree (see \Cref{subsection: backtrack}) in $\Sy_n$ by $n/(2p)$, using dual codes.   
Our second main complexity result will then follow immediately from \Cref{limit depth by base size} and \Cref{thm: norm sym for our H in expo klogk}.

\begin{theorem} \label{main theorem - cpk}
The \textsc{Norm-Sym} problem for $H$ in class $\mathfrak{InP}(\Cy_p)$ can be solved in time \linebreak $\mathrm{min}\left( 2^{O(\frac{n}{p} \log{\frac{n}{p}}+ \log{n})},  2^{O(\frac{n}{2p} \log{n})} \right)$. 
\end{theorem}

Current practical methods to solve \textsc{Norm-Sym} for groups in $\mathfrak{InP}(\Cy_p)$ are very slow.  
The implementation in the computer algebra system \texttt{GAP}~\cite{GAP4} struggles to compute $N_{\Sy_n}(H)$ even when $H$ has very small degree (median time of more than 10 minutes for degree around 30, see \Cref{section: results}). 
This paper also develops an effective practical algorithm to solve \textsc{Norm-Sym} for groups in $\mathfrak{InP}(\Cy_p)$, using the above ideas. 
In \Cref{section: results}, we provide evidence that our new algorithm performs far better than the one currently implemented in \texttt{GAP}.

The paper structure is as follows. In \Cref{section: general normalisers}, we give some background and present general results for computing $N_{\Sy_n}(H)$ for intransitive groups $H$. 
In \Cref{section: Cpk}, we prove \Cref{complexity main theorem} and in \Cref{section: reduce tree}, we shall prove \Cref{main theorem - cpk}.
In \Cref{subsection: pruning}, we present several techniques for speeding up a backtrack search to compute $N_{\Sy_n}(H)$ for $H  \in \InP(\Cy_p)$. 
We will describe our algorithm to compute $N_{\Sy_n}(H)$ for such $H$ in \Cref{section: algorithm}.
In \Cref{section: dihedral}, we extend our methods to class $\mathfrak{InP}(\Di_{2q})$, where $q$ is an odd prime.
Finally, in \Cref{section: results}, we present runtime data.


\section{Background and preliminaries}
\label{section: general normalisers}

In this section, we consider \textsc{Norm-Sym} for an arbitrary intransitive group $H \leq \Sy_n$ with orbits $\mathcal{O} = \{ \Omega_1, \Omega_2, \ldots, \Omega_k \}$. 
In \Cref{subsection: backtrack}, we give some preliminary complexity results and present a short overview of normaliser computation using backtrack search.
In \Cref{subsection: equiv orbs}, we discuss equivalent orbits, then in \Cref{section: Normalisers of subdirect products}, 
for a transitive group $A$, we construct a natural subgroup $L$ of $\Sy_n$ containing $N_{\Sy_n}(H)$ for $H \in \InP(A)$.  


\subsection{Permutation group algorithms}
\label{subsection: backtrack}

In this section, let $G \leq \Sy_n$ be given by a generating set $X$ and let $\Omega = \{ 1,2,\ldots, n\}$. It is well known that in time $O(|X|n^2 + n^5)$, we may replace $X$ by a generating set of size at most $n$ (see, for example, \cite{seress}).
Therefore, we assume that $|X| \leq n$ and we measure the complexity of permutation group problems in terms of $n$. 

A \emph{base} $B$ of $G$ is a tuple $(\beta_1 , \beta_2 , \ldots, \beta_m) \in \Omega^m$ such that the pointwise stabiliser $G_{(\beta_1 , \beta_2 , \ldots ,\beta_m)}$ is trivial. 
The \emph{base image} of an element $g$ of $G$ relative to $B$ is $B^g  \coloneqq  ( \beta_1^g , \beta_2^g , \ldots ,\beta_m^g )$.

For $G = \langle z_1, z_2, \ldots, z_k \rangle \leq \Sy_n$ and $K = \langle y_1, y_2, \ldots, y_l \rangle \leq \Sy_m$, a homomorphism $\phi:G\rightarrow K$ is \emph{given by generator images} if it is encoded by a list $[z_1, \ldots, z_k, y_1, \ldots, y_l, \phi(z_1), \ldots, \phi(z_k)]$.
We shall assume that all homomorphisms are given by generator images. 

The following results are standard and we refer the reader to \cite{handbookCGT, seress} for further details. 

\begin{lemma} \label{elem polytime}
Let $G = \langle X \rangle \leq \Sy_n$. 
\begin{enumerate}
    \item \label{cent in sym}  In polynomial time, we can: compute the orbits of $G$; compute the order $|G|$; compute the restriction $G|_{\Delta}$ of $G$ to a given $G$-invariant set $\Delta$; compute a base for $G$; compute the centraliser $C_{\Sy_n}(G)$.
    \item \label{im of homom and preim of isom} Given a permutation group $K$, a homomorphism $\phi: G \rightarrow K$, and permutations $g \in G$ and $k \in \mathrm{Im}(\phi)$, in polynomial time we can compute the image $\phi(g)$ and a pre-image of $k$. 
\end{enumerate}
\end{lemma}

We compute normalisers using backtrack search (see, for example, \cite{handbookCGT, seress} for more details).
Suppose that we want to compute $N_{G}(H)$.
The \emph{search tree $T$ of $G$ with respect to $B$} is a rooted tree of depth $m$, where the root node is labelled by the empty tuple $( \, )$ 
and nodes at depth $d$ are labelled with elements of the orbit $(\lists{\beta}{,}{d})^G$ such that each node $(\lists{\alpha}{,}{d})$ has parent $(\lists{\alpha}{,}{d-1})$. 
We can associate each node of $T$ with a coset of a point stabiliser in $G$ by defining
\[\Psi : \left( \lists{\alpha}{,}{d} \right) \mapsto \{ g \in G \mid \beta_i^g = \alpha_i,  \, \text{for all } 1 \leq i \leq d \}.\]
We traverse the search tree $T$ depth-first and gather the elements of $N_G(H)$ that we find in a group $N$, which is updated as search progresses. 

When solving problems with backtrack search, the runtime is correlated to the number of nodes in the search tree we visit. We reduce this number in two ways. 
Firstly, if there exists a proper subgroup $S$ of $G$ containing $N_G(H)$, we search in the tree of $S$ with respect to a base of $S$ instead.
We will describe how we can find such an $S$ for $H \in \InP(\Cy_p) \cup \InP(\Di_{2p})$ in \Cref{section: Normalisers of subdirect products}. 
Secondly, if we can deduce for some node $\tau$ of $T$ that $\Psi(\tau) \cap N_G(H)$ is empty, or that all elements of $\Psi(\tau)\cap N_G(H)$ is contained in the group $N$ of normalising elements we have already found, then we can skip traversing the subtree rooted at $\tau$. We call this skipping \emph{pruning} $T$. 

Methods for deducing that $\Psi(\tau) \cap N_G(H)$ is a non-empty subset of $N$ are known, see for example \cite[Section~9.1]{seress}, but we present new, efficient, techniques for deducing that $\Psi(\tau) \cap N_G(H) = \emptyset$ for $H \in \InP(\Cy_p)$ in \Cref{subsection: pruning,section: reduce tree}. 
We will also use the following elementary lemma (whose proof is clear) to deduce that $\Psi(\tau) \cap N_G(H) = \emptyset$. 

\begin{lemma} \label{prune by stabs or projections}
\begin{enumerate}
    \item \label{stabToStab} If there exists $\sigma \in N_{\Sy_n}(H)$ such that $(\lists{\delta}{,}{m})^\sigma = (\lists{\gamma}{,}{m})$, then ${(H_{(\lists{\delta}{,}{m})})}^\sigma = H_{(\lists{\gamma}{,}{m})}$. 
    \item \label{norm orbs to orbs with same projection} Let $\Delta_1$ and $\Delta_2$ be unions of $H$-orbits. 
If there exists $\sigma \in N_{\Sy_{n}}(H)$ such that $\Delta_1^\sigma = \Delta_2$, then $(H|_{\Delta_1})^\sigma = H|_{\Delta_2}$.
\end{enumerate}
\end{lemma}

Throughout the paper, for all subsets $\Delta$ of $\Omega$, we shall consider subgroups of $\Sym(\Delta)$ as subgroups of $\Sym(\Omega)$ with support $\Delta$. 


\subsection{Equivalence of orbits}
\label{subsection: equiv orbs}

In this subsection, we define an equivalence relation on the set $\mathcal{O}$ of $H$-orbits, and show how it is used in centraliser and normaliser computation.

\begin{definition} \label{defn: equiv orbs} \label{remark: bijection to permutation}
Two orbits $\Omega_i$ and $\Omega_j$ of $H$ are \emph{equivalent}, denoted $\Omega_i \equiv_H \Omega_j$, if  
there exists a bijection $\psi: \Omega_i \rightarrow \Omega_j$ such that 
\begin{equation} \label{eqn: def of equiv orbs}
\psi(\delta^h) = \psi(\delta)^h \quad \text{for all $h \in H$ and $\delta \in \Omega_i$. }
\end{equation}
We say that $\psi$ \emph{witness} the equivalence. 

For $\Omega_i$ and $\Omega_j$ in $\mathcal{O}$ and a bijection $\varphi : \Omega_i \rightarrow \Omega_j$, we denote by $\overline{\varphi}$ the involution in $\Sym(\Omega_i \cup \Omega_j)$ such that $\alpha^{\overline{\varphi}} = \varphi(\alpha)$ for all $\alpha \in \Omega_i$. 
Hence $\Omega_i$ and $\Omega_j$ are equivalent if and only if there exists an involution $\sigma = \overline{\psi} \in \Sym(\Omega_i  \cup \Omega_j)$ such that
\begin{equation} \label{equiv orb iff conjugates}
h|_{\Omega_j} = (h|_{\Omega_i})^{\sigma} \quad \text{for all $h \in H$. }
\end{equation}
\end{definition}

Next we see how the relation $\equiv_H$ helps us compute $C_{\Sy_n}(H)$ for intransitive groups $H$. 

\begin{lemma} [{\cite[\S6.1.2]{seress}}] \label{cent of subdir} 
Let $H\leq \Sy_n$ and let $\mathcal{B}_1,\mathcal{B}_2, \ldots, \mathcal{B}_a$ be the $\equiv_H$-classes. 
For $1 \leq i \leq a$, let $\mathcal{B}_i  \coloneqq  \{ \Omega_{i1}, \Omega_{i2}, \ldots, \Omega_{i|\mathcal{B}_i|} \}$, 
and for $2 \leq j \leq |\mathcal{B}_i|$, let $\psi_{ij}:\Omega_{i1} \rightarrow \Omega_{ij}$ witness the equivalence. 
Let $B_i  \coloneqq  \langle \overline{\psi_{ij}} \mid 2 \leq j \leq |\mathcal{B}_i| \rangle$ and let $C_i \coloneqq  C_{\Sym(\Omega_i)}(H|_{\Omega_i})$. 
Then 
\begin{eqnarray*}
C_{\Sy_n}(H) = \langle  C_1 \times C_2 \times \cdots \times C_k,  B_1 \times B_2 \times \cdots \times B_a  \rangle 
\cong \prod \limits_{i=1}^{t}  C_{\Sym(\Omega_{i1})}(H|_{\Omega_{i1}}) \wr \Sy_{|\mathcal{B}_i|}. 
\end{eqnarray*}
In particular, the elements $\overline{\psi_{ij}}$ and $C_{\Sy_n}(H)$ can be computed in polynomial time.
\end{lemma}


By \Cref{prune by stabs or projections}, we can deduce that there are no normalising elements in a subtree of $T$ by showing that some restrictions or stabilisers are not conjugate in $\Sy_n$.
We can use orbit equivalence to show that subgroups of $\Sy_n$ are not conjugate.

\begin{lemma} \label{conjn maps equiv class to equiv class}
Let $R$ and $Q$ be subgroups of $\Sy_n$ such that $Q = R^\sigma$ for some $\sigma \in \Sy_n$, and let $\mathcal{B}$ be an $\equiv_R$-class of orbits. 
Then $\{ \Delta^\sigma \mid \Delta \in \mathcal{B} \}$ is an $\equiv_Q$-class. 
Hence, with $H$ as in \Cref{cent of subdir}, the group $N_{\Sy_{n}}(H)$ acts on $\{ \mathcal{B}_1, \mathcal{B}_2, \ldots, \mathcal{B}_a\}$.
\end{lemma}

\begin{proof}
If $\Delta$ and $\Gamma$ are equivalent $R$-orbits,  
then by \Cref{equiv orb iff conjugates} there exists an involution $\mu$ in $ \Sym(\Delta\cup\Gamma)$ such that $r|_{\Gamma} = (r|_{\Delta})^\mu$ for all $r \in R$. 
Let $q \in Q$. Then there exists $r \in R$ such that $r^\sigma = q$.
So
\[ 
q|_{\Gamma^\sigma}= (r|_{\Gamma})^\sigma = (r|_{\Delta})^{\mu\sigma} 
= (r|_{\Delta})^{\sigma\sigma^{-1}\mu\sigma} =  (q|_{\Delta^\sigma})^{\sigma^{-1}\mu\sigma}, 
\]
and hence $\Delta^\sigma \equiv_Q \Gamma^\sigma$ by \Cref{equiv orb iff conjugates}.
\end{proof}

Next, we show that if $\equiv_H$ is not the equality relation then the computation of $N_{\Sy_n}(H)$ reduces in polynomial time to computing the normaliser of a group $H_1$ of a smaller degree for which $\equiv_{H_1}$ is equality.
For a partition $P$ of a set $S$, we denote by $[s]$, or $[s]_P$ if we wish to emphasise $P$, the cell of $P$ containing $s$. 
For $s,s' \in S$, we write $s \sim s'$ (or $s \sim_P s'$) to mean that $s$ and $s'$ are in the same cell of $P$, so $[s] = [s']$. 

\begin{proposition} \label{can remove equiv orbs}
\textsc{Norm-Sym} for $H \leq \Sy_n$ reduces in polynomial time to the special case where the $H$-orbits are pairwise inequivalent. 
\end{proposition}

\begin{proof}
We may assume that $H$ is intransitive and that $\equiv_H$ is not the equality relation. 
With the notation of \Cref{cent of subdir}, 
denote by $\mathcal{P}(H)$ the partition of $\{ \Omega_{11}, \Omega_{21}, \ldots, \Omega_{t1} \}$ such that $\Omega_{i1} \sim \Omega_{j1}$ if and only if $|\mathcal{B}_i| = |\mathcal{B}_j|$.
Let $\Gamma  \coloneqq  \bigcup_{i=1}^{t} \Omega_{i1}$ and let $U$ be the stabiliser of $\mathcal{P}(H)$ in $\Sym(\Gamma)$. 
We prove the claim by showing that, in polynomial time, one can construct a homomorphism $\theta: U \rightarrow \Sy_n$ such that $N_{\Sy_n}(H) = \langle \theta(N_U(H|_{\Gamma})),  C_{\Sy_n}(H) \rangle$. 

Let the bijections $\psi_{ij}: \Omega_{i1} \rightarrow \Omega_{j1}$ be as in  \Cref{cent of subdir} and let $\psi_{i1} = 1$. 
For all $u \in U$, define $\theta(u)$ by: for all $\alpha \in \Omega_{is}$, if $ \Omega_{i1}^{u} = \Omega_{j1}$ then $\alpha^{\theta(u)} = \alpha^{\overline{\psi_{is}}^{-1} u \overline{\psi_{js}}} \in \Omega_{js}$. 
One can check that $\theta(u)$ is indeed in $\Sy_n$ and that $\theta$ is a homomorphism.  

To see that $N_{\Sy_n}(H)$ contains $\langle \theta(N_U(H|_{\Gamma})),  C_{\Sy_n}(H) \rangle$, it suffices to show that it contains the image of $N_U(H|_{\Gamma})$ under $\theta$. Let $h \in H$ and $u \in N_U(H|_{\Gamma})$.
Then there exists an $h' \in H$ such that $h'|_{\Gamma} = (h|_{\Gamma})^{u}$. 
As $\overline{\psi_{j1}}=1$, for all $\alpha \in \Gamma$, the image $\alpha^{\theta(u)} = \alpha^{u}$ is in $\Gamma$. 
So $h^{\theta(u)}|_{\Gamma} = (h|_{\Gamma})^{u} =h'|_{\Gamma}$.  
Now for $1 \leq i \leq t$ and $2 \leq s \leq |\mathcal{B}_i|$, by the definition of $\overline{\psi_{is}}$,
\[h^{\theta(u)}|_{\Omega_{is}} = (h^{\theta(u)}|_{\Omega_{i1}})^{\overline{\psi_{is}}}
= (h'|_{\Omega_{i1}})^{\overline{\psi_{is}}} = h'|_{\Omega_{is}},
\]
so $h^{\theta(u)} = h'$. Hence $\theta(u) \in N_{\Sy_n}(H)$. 

To show the converse containment, let $\nu \in N_{\Sy_n}(H)$. 
By \Cref{conjn maps equiv class to equiv class}, $\nu$ acts on the $\equiv_H$-classes $\mathcal{B}_1, \mathcal{B}_2, \ldots, \mathcal{B}_a$.
By \Cref{cent of subdir}, $C_{\Sy_n}(H)$ induces the full symmetric group, independently, on each $\mathcal{B}_i$. 
So there exists a $c \in  C_{\Sy_n}(H)$ such that $\nu c$ fixes $\Gamma$ setwise, and so $(\nu c)|_{\Gamma} \in N_U(H|_{\Gamma})$. 
Let $\sigma \coloneqq \theta((\nu c)|_{\Gamma})$ and $h \in H$. Then $(h|_{\Gamma})^{\nu c} = (h|_{\Gamma})^{\sigma}$ and so by the argument of the previous paragraph $h^{\nu c} = h^{\sigma}$.
Therefore $\nu c \sigma^{-1} \in  C_{\Sy_n}(H)$ and hence $\nu \in \langle \theta(N_U(H|_{\Gamma})),  C_{\Sy_n}(H) \rangle$.

Since the permutations $\overline{\psi_{ij}}$ and $C_{\Sy_n}(H)$ can be computed in polynomial time by \Cref{cent of subdir}, the complexity claim is immediate. 
\end{proof} 

Hence in much of the remainder of the paper, we shall assume that the $\equiv_H$-relation is trivial. 


\subsection{Normalisers of groups in \texorpdfstring{$\InP(A)$}{InP(A)}: overgroups} 
\label{section: Normalisers of subdirect products}

We now assume that there exists an integer $m$ and a transitive subgroup $A$ of $\Sy_m$ such that $H$ is in $\InP(A)$. 

\begin{definition} \label{enveloping group}
A subgroup $R$ of $S = Q_1 \times Q_2 \times \cdots \times Q_k$ is a \emph{subdirect product} of $S$ if each projection of $R$ onto $Q_i$ is surjective. 

Let $G_i \coloneqq H|_{\Omega_i}$ for $1 \leq i \leq k$. 
Then $H$ is a subdirect product of $G = G_1 \times G_2 \times \cdots \times G_k$, where we identify the direct product with the corresponding subgroup of $\Sy_n$. We call $G$ the \emph{enveloping group} of $H$.
\end{definition}

Two permutation groups $R \leq \Sym(\Delta)$ and $S\leq \Sym(\Gamma)$ are \emph{permutation isomorphic} if there exists a bijection $\phi : \Delta \rightarrow \Gamma$ and an isomorphism $\psi: R \rightarrow S$ such that $\phi(\delta^r) = \phi(\delta)^{\psi(r)}$ for all $\delta \in \Delta$ and $r \in R$.
We say that $\phi$ \emph{witnesses} the permutation isomorphism from $R$ to $S$. 

Next we define a subgroup $L \leq \Sy_n$, analyse its structure and show that $L$ contains $N_{\Sy_n}(H)$. 

\begin{lemma} \label{define and analyse L}\label{def of subdirect of perm isom copies}{\label{normInWreath}}
Let $G = G_1 \times G_2 \times \cdots \times G_k$ be the enveloping group of $H$. 
For $2 \leq j \leq k$, let $\phi_{j}: \Omega_1 \rightarrow \Omega_j$ witness the permutation isomorphism from $G_1$ to $G_j$. 
For $1 \leq i \leq k$, let $N_i  \coloneqq  N_{\Sym(\Omega_i)}(G_i)$.
Let $B  \coloneqq  \langle N_1, N_2 , \ldots , N_k \rangle \leq \Sy_n$, let $K \coloneqq \langle \overline{\phi_j} \mid 2 \leq j \leq k  \rangle \leq \Sy_n$, and let $L \coloneqq  \langle  B,  K \rangle$. Then
\begin{enumerate}
    \item \label{K isom Sk} 
    $K$ acts faithfully as $\Sym(\mathcal{O})$ on the set $\mathcal{O}$ of $H$-orbits; 
    \item \label{item: L is wreath} $L$ is permutation isomorphic to $N_1 \wr \Sy_k$ in its imprimitive action;
    \item \label{norm of H in wreath} $N_{\Sy_n}(H) \leq L$;
    \item \label{norm to BK} in polynomial time, given $l \in L$,  we can compute $b \in B$ and $\kappa \in K$ such that $l = b\kappa$.
\end{enumerate}
\end{lemma}

\begin{proof}
\Cref{K isom Sk}:
Let $\xi: K \rightarrow \Sy_k$ be the permutation representation of $K$ on $\mathcal{O}$. Then $\xi$ is surjective since $\overline{\phi_i}$ induces the permutation $(\Omega_1,\Omega_i)$ on $\mathcal{O}$. 
For injectivity, let $\kappa \in K$ be such that $\xi(\kappa)=1$.
Then $\kappa$ setwise stabilises each $\Omega_i$. 
Fix $i$, let $\alpha \in \Omega_i$, and consider $\kappa$ as a word in the $\overline{\phi_j}$. 
Since $\Omega_i$ is moved only by $\overline{\phi_i}$, if $\overline{\phi_i}$ does not occur in $\kappa$ then $\alpha$ is fixed by $\kappa$, so assume otherwise. 
As $\alpha^\kappa$ is also in $\Omega_i$, there exists a subword $\kappa'$ of $\kappa$ such that $\alpha^{\kappa} = \alpha^{\overline{\phi_i} \kappa' \overline{\phi_i}}$, where $\alpha^{\overline{\phi_i}}$ and $\alpha^{\overline{\phi_i} \kappa'}$ are in $\Omega_1$. 
Now since each $\Omega_j$ is moved only by $\overline{\phi_j}$, 
there exist (not necessarily distinct) integers $l_1, l_2, \ldots, l_r \in \{2,3, \ldots, k\} $ such that $(\alpha^{\overline{\phi_i}})^{\kappa'} = (\alpha^{\overline{\phi_i}})^{ \overline{\phi_{l_1}}^2 \overline{\phi_{l_2}}^2 \ldots \overline{\phi_{l_r}}^2 } = \alpha^{\overline{\phi_i}} $.
Therefore $\alpha^\kappa = \alpha$. \\
\Cref{item: L is wreath}:
We first show that $L$ is imprimitive with block system $\mathcal{O}$.
Let $\alpha \in \Omega_i$ and $ \beta \in \Omega_j$.
Then $\alpha^{\overline{\phi_i}}$ and $ \beta^{\overline{\phi_j}}$ are points in $\Omega_1$, so there exists $g \in G \leq L$ such that $\alpha^{\overline{\phi_i}g} = \beta^{\overline{\phi_j}}$, and hence $L$ is transitive. Since $B$ and $K$ preserve $\mathcal{O}$, it follows that $\mathcal{O}$ is a block system for $L$. 

Consider the kernel $J$ of the action of $L$ on $\mathcal{O}$.
Since the action of $K$ on $\mathcal{O}$ is faithful, $J \leq B$. 
Conversely, since each $N_i$ fixes $\mathcal{O}$, the group $B$ is a subgroup of $J$. 
So $J \cong N_1 \times N_2 \times \cdots \times N_k$, and hence, up to isomorphism, $L \leq N_1 \wr \Sy_k$. 
But as $\Sy_k \cong K \leq L$, it follows that $L \cong N_1 \wr \Sy_k$. \\
\Cref{norm of H in wreath}: The normaliser $N_{\Sy_n}(H)$ is permutation isomorphic to a subgroup of $N_1 \wr \Sy_k$ in its natural imprimitive action \cite[\S11]{hulpkeTransitive}, so this follows from \Cref{item: L is wreath}.  \\
\Cref{norm to BK}: 
The element $l$ induces a permutation $\sigma$ of the set $\mathcal{O}$ of orbits. 
Then $\kappa = \xi^{-1}(\sigma)$ can be computed in polynomial time by \Cref{elem polytime}.\ref{im of homom and preim of isom}, and $l \kappa^{-1}$ fixes each $H$-orbit, so is in $B$. 
\end{proof}


\section{\texorpdfstring{$\InP(\Cy_p)$}{InP(Cp)} and automorphisms of codes}
\label{section: Cpk}\label{subsection: group to code}

In this section, let $p$ be prime. 
For the rest of this section, we shall assume that the following hypothesis holds.

\begin{hypothesis}\label{defining H as subdir of  Cpk} \label{cpk hypo}
Let $n=pk$ and let $H$ be a subgroup of $\Sy_n$ in class $\mathfrak{InP}(\Cy_p)$. 
Let $\mathcal{O}$ $\coloneqq \{\Omega_1, \Omega_2, \ldots, \Omega_k \}$ be the orbits of $H$, ordered such that $|H| = |(H|_{\cup_{i \leq s}\Omega_i})| =  p^s$ for some $s$. 

Let $G$ be the enveloping group of $H$ and let $g_1$ be a permutation in $\Sym(\Omega_1)$ generating $G_1$. For $1 \leq i \leq k$, let $\phi_i: \Omega_1 \rightarrow \Omega_i$ witness the permutation isomorphism between $G_1$ and $G_i$, and let $g_i = g_1^{\overline{\phi_i}}$. 
\end{hypothesis}

We now set up an isomorphism $\gamma$ from $H$ to a linear code. Then, in \Cref{subsection: norm as aut}, we shall prove that computing $N_{\Sy_n}(H)$ is polynomial-time equivalent to computing the monomial automorphism group of $\gamma(H)$.


Denote the set of all $s \times k$ matrices over the field $\mathds{F}_p$ by $\mathrm{M}(s,k,p)$. 
For $M \in \mathrm{M}(s,k,p)$, we denote by $M_{i,*}$ and $M_{*,j}$ the $i$-th row and the $j$-th column of $M$ respectively. 
For a tuple $I$ of distinct elements of $\{1,2, \ldots, s\}$, we denote by $M_{I,*}$ the matrix of dimension $|I| \times k$ such that the $i$-th row of $M_{I,*}$ is $M_{I_i,*}$, and similarly for columns. 
We denote the row space of $M$ by $\langle M \rangle$.

A \emph{linear code} $C$ over $\mathds{F}_p$ is a subspace of $\mathds{F}_p^k$ for some $k$, denoted by $C \leq \mathds{F}_p^k$. 
A \emph{generator matrix} $M$ of $C\leq \mathds{F}_p^k$ is a matrix in $\mathrm{M}(s,k, p)$ for some $s$ whose rows form a basis for $C$. 
We shall assume that all linear codes are given by generator matrices.
The matrix $M$ is in \emph{standard form} if its first $s$ columns form the identity matrix $\mathrm{I}_s$. 
For more information on linear codes, refer to, for example, \cite{lintCoding}. 
Since throughout the paper we shall be moving between exponential and additive notation, we shall identify elements of $\mathds{F}_p$ with integers $\{0,1, \ldots, p-1 \}$. 
We denote by $\mathds{F}_p^*$ the multiplicative group of $\mathds{F}_p$.

\begin{definition} \label{def: mapping direct prod of cyclic group to field}
Let $\gamma$ be the isomorphism from $G$ to $\mathds{F}_p^k$ defined by
\[
\gamma(g_1^{r_1} g_2^{r_2} \ldots g_k^{r_k})= (\lists{r}{,}{k}) . 
\]
\end{definition}

Observe that $\gamma(H)$ is a subspace of $\gamma(G) = \mathds{F}_p^k$, so $\gamma(H)$ is linear code of length $k$ and dimension $s$ over $\mathds{F}_p$. 
Our assumption that $|(H|_{\cup_{i \leq s} \Omega_i})| = p^s$ means that no column permutations are needed to put a generator matrix of $\gamma(H)$ in standard form, so let $M$ be such a generator matrix.

We show that, in polynomial time, we can decide if a given subgroup of $\Sy_n$ is in $\mathfrak{InP}(\Cy_p)$, and if so compute some accompanying data.

\begin{lemma} \label{detect H in InP(Cp) in poly time}\label{gi and phii in poly time}
Let $Q = \langle Y \rangle$ be a subgroup of $\Sy_n$ with $k$ orbits. Then in polynomial time, we can decide if $Q \in \InP(\Cy_p)$. 
Furthermore, for a subgroup $H = \langle X \rangle$ of $\Sy_n$ in $\InP(\Cy_p)$, in polynomial time, we can compute: 
an ordering $\Omega_1, \Omega_2, \ldots, \Omega_k$ of $\mathcal{O}$, the bijections $\phi_i$ and the generators $g_i$ from \Cref{cpk hypo}; the isomorphism $\gamma$ from \Cref{def: mapping direct prod of cyclic group to field}; and a generator matrix $M \in \mathrm{M}(s,k,p)$ for $\gamma(H)$ in standard form. 
\end{lemma}

\begin{proof}
The group $Q \in \InP(\Cy_p)$ if and only if for all $Q$-orbits $\Delta$, the size of $\Delta$ is $p$ and $|(Q|_{\Delta})|=p$. So we can decide if $Q \in \InP(\Cy_p)$ in polynomial time by \Cref{elem polytime}. 

The ordering of the $H$-orbits can be obtained in polynomial time by \Cref{elem polytime}.\ref{cent in sym}. 
Choices for $g_i$ can be computed in polynomial time since any non-trivial restriction $x|_{\Delta}$ of $x \in X$ generates $H|_{\Delta}$. 
Since conjugacy of permutations in symmetric groups can be computed in polynomial time, so can the bijections $\phi_i$. 
We construct $\gamma$ by mapping each $g_i$ to the $i$-th standard basis vector of $\mathds{F}_p^k$. Finally, $M$ can be obtained by finding a row reduced basis for $\langle \gamma(x) \mid x \in X \rangle$, in time polynomial in $s \in O(n)$ and $k \in O(n)$~\cite{linearAlgebra}. 
\end{proof}


\subsection{The normaliser of \texorpdfstring{$H$}{H} as an automorphism group of a linear code}
\label{subsection: norm as aut}

In this subsection we prove \Cref{norm of our groups polytime equiv to code aut}, but first we define the actions of certain subgroups of $\mathrm{GL}_k(p)$ on $\gamma(H)$.  

Let $D$ and $P$ be the groups of all diagonal and  permutation matrices in $\mathrm{GL}_k(p)$, respectively, and let $W  \coloneqq  \langle D, P \rangle$.
The natural action of $W$ on $\mathds{F}_p^k$ is called the \emph{monomial action}.
Observe that $D \cong (\mathds{F}_{p}^*)^k$, $P \cong \Sy_k$ and $W  = D \rtimes P \cong {\mathds{F}_{p}^*} \wr \Sy_k$. 

\begin{definition} \label{defn: monomial aut}
Let $C\leq \mathds{F}_p^k$ be a linear code. 
The \emph{monomial automorphism group} $\MAut(C)$ of $C$ is the subgroup of $W$ that setwise stabilises $C$.
Two codes $C,C'\leq \mathds{F}_p^k$ are \emph{monomially equivalent} if $C w=C'$ for some $w \in W$  (we denote vector-matrix multiplication by concatenation). 
\end{definition}

Let $L = B \rtimes K$ be as in \Cref{normInWreath}. 
Observe that $N_{\Sy_p}(\Cy_p) = \mathrm{AGL}_1(p) \cong \Cy_p \rtimes \Cy_{p-1}$, so $L \cong (\Cy_p \rtimes \Cy_{p-1}) \wr \Sy_k$. We shall show that the conjugation action of $L/G \cong \Cy_{p-1} \wr \Sy_k$ on $G$ is equivalent to the action of $W$ on $\mathds{F}_p^k$, and hence show that computing $N_{\Sy_n}(H)$ is polynomial-time equivalent to computing $\MAut(\gamma(H))$. 

\begin{lemma}\label{conjActionToActionInVector}\label{mapping from $N$ to $(F_p^*)^k$} \label{notation: diag and perm}
Define a homomorphism $\rho : K \rightarrow P$ by $\rho(\kappa)_{i,j} =1$ if and only if $\Omega_i^{\kappa} = \Omega_j$. 
Let $\zeta : B \rightarrow D$ map
$b$ to $diag(d_1, d_2, \ldots, d_k)$ where $b^{-1} g_ib  = g_i^{d_i}$  for $1 \leq i \leq k$. 
Define a map $\Xi : L  \rightarrow W $ by writing each $l \in L$ as $b\kappa$ for some $b \in B$ and $\kappa \in K$ and mapping 
\[l = b\kappa  \mapsto \zeta(b) \rho(\kappa). 
\]
Then the following statements hold. 
\begin{enumerate}
    {\item \label{xi is epi with kernel $G$}$\Xi$ is an epimorphism with kernel $G$.}
   \item \label{equiv action}  $\gamma(g^{l}) = \gamma(g){\Xi(l)}$ for all $g \in G$ and $l\in L$. 
    \item \label{induces homom of norm and kernel} $\Xi(N_{\Sy_n}(H)) = \MAut(\gamma(H))$, so $N_{\Sy_n}(H)$ is the full pre-image of $ \MAut(\gamma(H))$ under $\Xi$. 
    \item \label{can compute im and preim of Xi} Given $l \in L$ and $w \in W$, we can compute $\Xi(l)$ and a preimage of $w$ under $\Xi$, in time polynomial in $n$. 
\end{enumerate}
\end{lemma}

\begin{proof}
{\Cref{xi is epi with kernel $G$}:} Let $b_1, b_2 \in B$ and $\kappa_1, \kappa_2 \in K$.
Then \[\Xi(b_1 \kappa_1 b_2 \kappa_2) = \Xi(b_1 b_2^{\kappa_1^{-1}} \kappa_1 \kappa_2) = \zeta(b_1) \zeta(b_2^{\kappa_1^{-1}})\rho(\kappa_1) \rho(\kappa_2), \]
so to show that $\Xi$ is a homomorphism, it suffices to show that 
$v \rho(\kappa_1) \zeta(b_2){\rho(\kappa_1)^{-1}} = v\zeta(b_2^{\kappa_1^{-1}}) $ for all $v \in \mathds{F}_p^k$. 
For $i \in \{1,2, \ldots, k\}$, let $\Omega_j = \Omega_i^{\kappa_1}$, and let $d_j = \zeta(b_2)_j$.
Then 
\[ 
(v \rho(\kappa_1) \zeta(b_2){\rho(\kappa_1)^{-1}})_{i} = (v \rho(\kappa_1) \zeta(b_2))_{j} = (v \rho(\kappa_1))_{j}d_j = v_{i}d_j.
\]
Since $g_i^{\kappa_1 b_2 \kappa_1^{-1}} = g_j^{b_2 \kappa_1^{-1}} = g_j^{d_j \kappa_1^{-1}} =  g_i^{d_j}$, 
\[
(v \zeta(b_2^{\kappa_1^{-1}}))_{i} = v_{i} (\zeta(b_2^{\kappa_1^{-1}}))_i = v_{i} d_j. 
\]
So $\zeta(b_2^{\kappa_1^{-1}}) = \zeta(b_2)^{\rho(\kappa_1)^{-1}}$, 
and hence $\Xi$ is a homomorphism. 
Since $\zeta$ is an epimorphism, so is $\Xi$. 
Finally, as $\mathrm{Ker}(\rho)$ is trivial, $\Ker(\Xi) = \Ker(\zeta) = G$. \\
\Cref{equiv action}: 
Let $l= b \kappa$ with $b \in B$ and $\kappa \in K$. 
Then there exist $r_i \in \mathds{F}_p$ and $d_i \in \mathds{F}_p^*$ such that $g|_{\Omega_i} = g_i^{r_i}$ and $g_i^{b} = g_i^{d_i}$. 
Let $\Omega_j = \Omega_i^{\kappa} $. 
Then $(g^{l})|_{\Omega_j} =  ( (g^{b})|_{\Omega_i} )^{\kappa} =  (g_i^{r_id_i})^\kappa$. 
Recall the involutions $\overline{\phi_i} \in \Sym(\Omega_1 \cup \Omega_i)$ from \Cref{defining H as subdir of  Cpk}. 
Since $K$ is generated by the $\overline{\phi_i}$ and conjugation by ${\overline{\phi_i}}$ swaps $g_1$ and $g_i$, it follows that $g_i^{\kappa} = g_j$.
So $(g^{l})|_{\Omega_j} = g_j^{r_id_i}$ and $\gamma(g^{l})_j = r_id_i$. Hence for all $j$, 
\[ 
(\gamma(g){\Xi(l)})_j  
= (\gamma(g){\zeta(b)\rho(\kappa) })_j  
= (\gamma(g){\zeta(b)})_i = r_id_i = \gamma(g^{l})_j.  
\]
\Cref{induces homom of norm and kernel}: 
First note that $N_{\Sy_n}(H) \leq L$ by \Cref{normInWreath}.\ref{norm of H in wreath}. 
By \Cref{equiv action}, if $l \in L$ is in $N_{\Sy_n}(H)$, then $\gamma(h)\Xi(l) \in \gamma(H)$ for all $h \in H$ and so $\Xi(l) \in \MAut(\gamma(H))$. 
Conversely for $w \in \MAut(\gamma(H))$, by \Cref{xi is epi with kernel $G$}, there exists an $l \in L$ such that $\Xi(l)= w$. Then $\gamma(h^{l}) = \gamma(h){w} \in \gamma(H)$ for all $h \in H$. So $l \in N_{\Sy_n}(H)$.  \\
\Cref{can compute im and preim of Xi}: By \Cref{normInWreath}.\ref{norm to BK}, in polynomial time, we can find $b \in B$ and $\kappa \in K$ such that $l=b\kappa$. 
For $1 \leq i \leq k$, we can find $d_i \leq p-1$ such that $b^{-1}g_ib = g_i^{d_i}$, and $j$ such that $\Omega_i^{\kappa} = \Omega_j$. As $k \leq n$, the images $\zeta(b)$ and $\rho(\kappa)$, and hence $\Xi(l)$, can be computed in time polynomial in $n$. 

To show that we can find an element with $\Xi$-image $w$ in polynomial time, for $1 \leq i \leq k$, let $d_i$ be the non-zero entry of $w_{i,*}$, and let $d = diag(d_1, d_2, \ldots, d_k)$. So in time polynomial in $k\log{p} \in O(n)$, we can compute $d \in D$ and $q \coloneqq  d^{-1} w \in P$ such that $w= d q$.

Now we find an element $\sigma$ of $\Sy_n$ such that $\sigma^{-1}g_i\sigma = g_i^{d_i}$, in time polynomial in $n$ since the $g_i$ have disjoint supports and conjugation in $\Sym(\Omega_i)$ can be solved in time polynomial in $|\Omega_i|$. Then $\sigma \in B$ and $\zeta(\sigma) = d \in D$.
Next, in time polynomial in $k$, we construct the element $\sigma$ of $\Sy_k$ such that the image $i^\sigma$ is the position of the non-zero entry in $q_{i,*}$. 
Letting $\xi$ be as in the proof of \Cref{define and analyse L}.\ref{K isom Sk}, $\kappa \coloneqq \xi^{-1}(\sigma)$ is the element of $K$ with $\rho(\kappa) = q$, which can be computed in time polynomial in $n$ by \Cref{elem polytime}.\ref{im of homom and preim of isom}. Therefore, in time polynomial in $n$, we can compute an element $\sigma \kappa$ of $L$ with $\Xi$-image $w$. 
\end{proof}

Finally we prove \Cref{norm of our groups polytime equiv to code aut}. 

\begin{proof} [Proof of \Cref{norm of our groups polytime equiv to code aut}]
To reduce \textsc{Norm-Sym} for groups $H \leq \Sy_{pk}$ in class $\InP(\Cy_p)$ to \textsc{MAut}, first notice that by \Cref{gi and phii in poly time}, in polynomial time, we can compute the enveloping group $G$ of $H$ and an isomorphism $\gamma: G \rightarrow \mathds{F}_p^k$.
Assume that we can compute a generating set $Y$ for $\MAut(\gamma(H))$ in time polynomial in $k$. 
Then by \Cref{conjActionToActionInVector}, $N_{\Sy_n}(H) = \langle \{ \Xi^{-1}(y) \mid y \in Y\}, G \rangle$, where each $\Xi^{-1}(y)$ denotes a pre-image of $y$ under $\Xi$.

For the backward reduction, let $C\leq \mathds{F}_p^{k}$ be a linear code given by a generator matrix $M \in \mathrm{M}(s,k, p)$.
Let $g_i=(p(i-1)+1, \ldots , pi) \in \Sy_{pk}$ for $1 \leq i \leq k$ and let $G = \langle g_1, g_2, \ldots, g_k \rangle$. 
Let $\gamma$ be as in \Cref{def: mapping direct prod of cyclic group to field}
and let $H = \langle \gamma^{-1}(M_{i,*}) \mid 1 \leq i \leq s \rangle $.
Assume that we can compute a generating set $Y$ for $N_{\Sy_{pk}}(H)$ in time polynomial in $k$. 
Then by \Cref{conjActionToActionInVector}.\ref{induces homom of norm and kernel}, $\MAut(C) = \Xi(N_{\Sy_{pk}}(H)) = \langle \Xi(y) \mid y \in Y\rangle$, which can be computed in time polynomial in $k$ by \Cref{conjActionToActionInVector}.\ref{can compute im and preim of Xi}.  

For the equivalence of \textsc{Conj-Sym} for groups in class $\InP(\Cy_p)$ and  \textsc{MEq} for codes over $\mathds{F}_p$, 
let $H_1$ and $H_2$ be subgroups of $\Sy_{pk}$ and let $C_1 = \gamma(H_1)$ and $C_2 = \gamma(H_2)$ be codes of length $k$ over $\mathds{F}_p$. It follows from \Cref{conjActionToActionInVector}.\ref{xi is epi with kernel $G$}--\ref{equiv action} that $H_1$ and $H_2$ are conjugate in $\Sy_n$ if and only if $C_1$ and $C_2$ are monomially equivalent. The rest of the proof is similar to that of the equivalence of \textsc{Norm-Sym} and \textsc{MAut}. 
\end{proof}

\begin{corollary} \label{cor: rel complexity}
Fix a prime $p$. Denote \textsc{MAut} for $p$-ary codes by $\textsc{MAut}_p$. Then the graph isomorphism problem is polynomial-time reducible to $\textsc{MAut}_p$, which is polynomial-time reducible to \textsc{Norm-Sym}. 
\end{corollary}

\begin{proof}
The graph isomorphism problem is polynomial-time reducible to \textsc{MEq} for $p$-ary codes \cite{codeEquivGI}, which is polynomial-time reducible to $\textsc{MAut}_p$ by \Cref{norm of our groups polytime equiv to code aut}. The result now follows as \textsc{Norm-Sym} for $\InP(\Cy_p)$ is a special case of \textsc{Norm-Sym}. 
\end{proof}


\section{Complexity results}
\label{section: reduce tree}

Let $H$ be as in \Cref{cpk hypo}, and let $L= B \rtimes K$ be as in \Cref{normInWreath}. 
In \Cref{subsection: limit depth}, we show that $N_{\Sy_n}(H)$ can be computed in time $2^{O(\frac{n}{2p} \log{n} )}$. 
In \Cref{subsection: canon im}, we show that for each $\kappa \in K$, if there exists $b \in B$ such that $b \kappa \in N_{\Sy_n}(H)$, then we can find such a $b$ in polynomial time. 
In \Cref{subsection: normFixingOrbs}, we first show that $N_{B}(H)$ can be computed in polynomial time, then we see how the results in this section come together to reduce the search space for $N_{\Sy_n}(H)$ to searching in $K \cong \Sy_k$ instead of $L \cong (\Cy_p \rtimes \Cy_{p-1}) \wr \Sy_k$, and hence prove \Cref{main theorem - cpk}.

\subsection{Limiting the depth of the search tree}
\label{subsection: limit depth}\label{defn: dual code}

In this subsection, we show how we can reduce the depth of the search tree using a possibly smaller group $H^{\bot}$, which we define now. 

The \emph{dual} code $C^{\bot}$ of a code $C\leq \mathds{F}_p^k$ is the subspace $\{ v \in \mathds{F}_p^k \mid  v \cdot c = 0 \text{ for all } c \in C \}$, where~`$\cdot$' denotes the standard dot product. 
Let $H^{\bot} \leq \Sy_n$ be $\gamma^{-1}(\gamma(H)^{\bot})$.
We shall show that there exists a bijection between $N_{\Sy_n}(H)$ and $N_{\Sy_n}(H^{\bot})$.
Recall from \Cref{normInWreath} that $N_{\Sy_n}(H) \leq L$.

\begin{lemma}  \label{search in dual}
Let $b \in B$ and $\kappa \in K$.
Then $b\kappa \in N_{\Sy_n}(H)$ if and only if  $b^{-1} \kappa \in N_{\Sy_n}(H^{\bot})$.
\end{lemma}

\begin{proof}
Let the epimorphism $\Xi: L \rightarrow W$ be as in \Cref{conjActionToActionInVector}.\\
$\Rightarrow$: 
Assume that $l=b\kappa \in N_{\Sy_n}(H)$ and let $\eta \in H^{\bot}$.
Let $\zeta(b ) = diag(d_1, d_2, \ldots, d_k) \in D$. 
We show that $\eta^{b^{-1}  \kappa} \in H^{\bot}$ by showing that $\gamma(h) \cdot  \gamma(\eta^{b^{-1} \kappa}) = 0$ for all $h \in H$, so let $h \in H$ and $g \coloneqq h^{l^{-1}} \in H$. 
Then 
\[
 0 = \gamma(g) \cdot \gamma(\eta) = \sum\nolimits_{i=1}^{k} \gamma(g)_i \gamma(\eta)_i =  \sum\nolimits_{i=1}^{k} \gamma(g)_i d_i \gamma(\eta)_i d_i^{-1} = \gamma(g){\zeta(b)} \cdot \gamma(\eta){\zeta(b^{-1})}.   
\]
Since $P$ permutes the coordinates of $\mathds{F}_p^k$, the product
$\gamma(g){\Xi(l) } \cdot \gamma(\eta){\Xi(b^{-1} \kappa)}$ is also $0$. 
Now by {\Cref{conjActionToActionInVector}.\ref{equiv action},} 
\[ 
\gamma(h) \cdot  \gamma(\eta^{b^{-1} \kappa}) 
= \gamma(g^{l}) \cdot \gamma(\eta^{b^{-1} \kappa})
= \gamma(g){\Xi(l) } \cdot \gamma(\eta){\Xi(b^{-1} \kappa)}
= 0. 
\]
$\Leftarrow$: The fact that $(H^{\bot})^{\bot} = H$ implies that $b \kappa = (b^{-1})^{-1} \kappa \in N_{\Sy_n}(H)$. 
\end{proof}

Now we prove a more precise version of the upper bound in \Cref{main theorem - cpk}.

\begin{proposition} \label{limit depth by base size}
Let $H = \langle X \rangle$ be a subgroup of $\Sy_n$ in class $\InP(\Cy_p)$, with $n=pk$ and $|H| = p^s$, and let $m  \coloneqq  \mathrm{min}\{ s, k- s\}$. 
Then $N_{\Sy_n}(H)$ can be computed in time $2^{O((m+1) \log{n} )}$. 
Hence $N_{\Sy_n}(H)$ can be computed in time $2^{O(\frac{n}{2p} \log{n})}$.
\end{proposition}

\begin{proof}
Notice that $|H^{\bot}| = p^{k-s}$, so $m$ is the minimum of $\log_p(|H|)$ and $\log_p(|H^{\bot}|)$. 
We shall show that $N_{\Sy_n}(H)$ can be computed in time $2^{O((s+1) \log{n} )}$.  
Then if $s>k/2$, using Lemmas \ref{define and analyse L}.\ref{norm to BK} and \ref{search in dual}, we may compute $N_{\Sy_n}(H)$ in time $2^{O((m+1) \log{n})}$ by computing $N_{\Sy_n}(H^{\bot})$ instead. 
From here, since $m \leq k/2$, the last assertion will follow. 

By \Cref{can remove equiv orbs}, without loss of generality, we may assume that $\equiv_H$ is the equality relation. 
By \Cref{detect H in InP(Cp) in poly time}, in polynomial time, we can check that $H \in \InP(\Cy_p)$, and obtain a generator matrix $M$ of $\gamma(H)$ in standard form. 
If $w \in W$ is in $\MAut(\gamma(H))$, then $M' \coloneqq Mw$ satisfies $\langle M \rangle = \langle M' \rangle$. 
Since $M$ is in standard form, the elements of $\langle M \rangle$ are uniquely determined by their first $s$ coordinates.
So to test whether a given $w \in W$ is in $\MAut(\gamma(H))$, it suffices to describe which columns of $M$ are mapped onto the first $s$ columns of $M'$ and how they are scaled, then the rest of the action of $w$ is determined.  We find all such $w  = diag(d_1, d_2, \ldots, d_k) \rho(\kappa)$ in $\MAut(\gamma(H))$ by considering all choices of $J = (1^{\kappa^{-1}}, 2^{\kappa^{-1}}, \ldots, s^{\kappa^{-1}}) $ and $v \coloneqq (d_i)_{i \in J} \in (\mathds{F}_p^*)^s$. 

For each such $J$ and $v$, let $M'$ be the partially defined matrix with $M'_{*, i}=d_{i^{\kappa^{-1}}} M_{*, i^{\kappa^{-1}}}$, for $1 \leq i \leq s$. Then each partial row of $M'$ extends to a unique vector in $\langle M \rangle$, which must be the corresponding row of $M'$. 
If $w \in \MAut(\gamma(H))$ then $(Mw)_{*,j}$ is a scalar multiple of a column of $M$. 
Therefore we test whether $J$ and $v$ yield an element of $\MAut(\gamma(H))$ by testing, for all $j >s$, if there exists a $j^{\kappa^{-1}} \not \in \{1^{\kappa^{-1}}, 2^{\kappa^{-1}}, \ldots, s^{\kappa^{-1}} \}$ and $d_{j^{\kappa^{-1}}}$ such that $M_{*,j}' = d_{j^{\kappa^{-1}}} M_{*, j^{\kappa^{-1}}}$.

If for some $j$ there are no such $M_{*, j^{\kappa^{-1}}}$ and $d_{j^{\kappa^{-1}}}$, then this choice of $J$ and $v$ does not extend to an element of $\MAut(\gamma(H))$, and we move on to the next.
Note that since the $H$-orbits are pairwise inequivalent, for each $j$, there are at most one possible choice for $j^{\kappa^{-1}}$.
If we have succeeded in choosing $j^{\kappa^{-1}}$ and $d_{j^{\kappa^{-1}}}$ for all $j$, then
\[
(M diag(d_{1^{\kappa^{-1}}}, \ldots, d_{k^{\kappa^{-1}}}) \rho(\kappa))_{*,j} = d_{j^{\kappa^{-1}}} M_{*, j^{\kappa^{-1}}}  = M'_{*,j} \quad \text{for $1 \leq j \leq k$,} 
\]
and so $w=diag(d_{1^{\kappa^{-1}}}, \ldots, d_{k^{\kappa^{-1}}}) \rho(\kappa) \in \MAut(\gamma(H))$. 
There are $k(p-1) \leq n$ choices of $d_{j^{\kappa^{-1}}} M_{*, j^{\kappa^{-1}}}$ for each $j \in \{ s+1, \ldots, k \}$, so for each $J$ and $r$, this step can be done in polynomial time. \medskip \\
Let $Y$ be the set of all elements $w \in \MAut(\gamma(H))$ found as above. 
We claim that $N_{\Sy_n}(H) = \langle C_{\Sy_n}(H), \Xi^{-1}(Y) \rangle$, where $\Xi^{-1}(Y)$ is any preimage of $Y$. 
To show that $N_{\Sy_n}(H) \leq \langle C_{\Sy_n}(H), \Xi^{-1}(Y) \rangle$, let $\nu \in N_{\Sy_n}(H)$. Then there exists $y \in Y$ such that $M\Xi(\nu) = My$, so $\Xi(\nu) y^{-1}$ stabilises $M$ and hence fixes each $v \in \gamma(H)$. 
So by \Cref{notation: diag and perm}.\ref{equiv action}, $\Xi^{-1}(\Xi(\nu)y^{-1}) = \nu \Xi^{-1}(y^{-1}) \in C_{\Sy_n}(H)$, therefore  $\nu \in  \langle C_{\Sy_n}(H), \Xi^{-1}(Y) \rangle$. 
The converse containment is clear. 


For the complexity claim, since there are $O(k^s (p-1)^s ) \in O(n^s)$ choices for $J$ and $v$, 
the set $Y$ can be computed in time $2^{O(s \log{n} + \log{n} )}$. 
Now by \Cref{notation: diag and perm}.\ref{can compute im and preim of Xi}, the set $\Xi^{-1}(Y)$ can be computed in $2^{O((s+1) \log{n} )}$ time, and by \Cref{elem polytime}.\ref{im of homom and preim of isom}, $C_{\Sy_n}(H)$ can be computed in time $2^{O(\log{n})}$. Therefore $N_{\Sy_n}(H)$ can be computed in time $2^{O((s+1) \log{n} )}$. 
\end{proof}

While searching for elements of $\MAut(\gamma(H))$ in $W$, the above result limits the depth of the search tree. In practice, we search for $\MAut(\gamma(H))$ in $P \cong \Sy_k$, as we shall see in \Cref{subsection: normFixingOrbs}.

\subsection{Normalising elements that act non-trivially on \texorpdfstring{$\mathcal{O}$}{O}}
\label{subsection: canon im}

Let $L = B  \rtimes K$ be as in \Cref{normInWreath} and recall that $N_{\Sy_n}(H) \leq L$. 
In \Cref{suffices to search in Sn}, we show that given a $\kappa \in K$, we can decide in polynomial time if there exists an element of $N_{\Sy_n}(H)$ which induces the same permutation as $\kappa$ on the set $\mathcal{O}$ of $H$-orbits.

\begin{lemma} \label{$w$ gives norm iff same semicanon}
Let $F  \coloneqq  \mathrm{GL}_s(p) \times D$.
Define an action of $F$ on $\mathrm{M}(s,k,p)$ by 
\[M^{(R,d)} = R^{-1}Md  \quad \text{for all $R \in \mathrm{GL}_s(p)$ and $d \in  D$}. \] 
Let $\kappa \in K$ and let $M, M' \in \mathrm{M}(s,k , p)$ be generator matrices of $\gamma(H)$ and $\gamma(H^{\kappa^{-1}})$ respectively. 
Then there exists a $b \in B$ such that $b \kappa \in N_{\Sy_n}(H)$ if and only if 
there exists an $f \in F$ such that $M^f = M'$.
\end{lemma}

\begin{proof}
Such a $b$ exists if and only if 
$H$ and $H^{\kappa^{-1}}$ are in the same $B$-orbit, or equivalently, by \Cref{conjActionToActionInVector}, $\gamma(H)$ and $\gamma(H^{\kappa^{-1}})$ are in the same $D$-orbit.  
The result now follows from the fact that $ \langle M \rangle = \langle M' \rangle$ if and only if $M$ and $M'$ are in the same orbit under left multiplication by $\mathrm{GL}_s(p)$. 
\end{proof}

We shall decide if $M' \in M^F$ by computing certain $F$-orbit representatives, which we now define. 
For $A, A' \in \mathrm{M}(s,k,p)$, let $A_{*,i}^R$ denote $A_{*,i}$ reversed. 
We define $A \prec A'$ if there exists a $j$ such that $A_{*,i} = A'_{*,i}$ for $1 \leq i < j$ and
$A_{*,j}^R <_{lex} {A'}_{*,j}^R$. 
We choose the representative of $A^F$ to be the least element under the ordering $\prec$.

Feulner gives an algorithm to compute such $F$-orbit representatives \cite[Algorithm~1]{feulner}. Since we will be proving its complexity, we will present our minor variation on the algorithm here.
A key ingredient is the support partition, which we define next. 
We denote the tuple $(1,2, \ldots, i)$ by $\overline{i}$. 
The \emph{support} of $v \in \mathds{F}_p^a$ is $\Supp(v) \coloneqq  \{i \mid  1 \leq i \leq {a}, \, v_i \neq 0 \}$, and the support of $V \leq \mathds{F}_p^a$ is $\Supp(V) = \cup_{ v \in V} \Supp(v)$. 

\begin{definition}\label{feulner's supp partn}
Let $M \in \mathrm{M}(s, k, p)$ be in standard form. 
For $1 \leq j \leq {k}$, 
the \emph{support partition} $\mathcal{Q}_j$ is the finest partition of $\{1,2, \ldots, s \}$ such that for $1 \leq i \leq {j}$, there exists a cell $Q$ of $\mathcal{Q}_j$ that contains $\Supp(M_{*,i})$.
\end{definition}

We present a simplified version of \cite[Algorithm~1]{feulner} as \Cref{algm: semicanon}. For $1 \leq j \leq k$, let $D^{[j]}$ be the subgroup of $D$ consisting of matrices of the form $diag(d_1, d_2, \ldots ,d_j, 1, \ldots, 1)$, where $d_i \in \mathds{F}_p^*$. 
\Cref{algm: semicanon} iteratively computes the orbit representative under $\mathrm{GL}_j(p) \times D^{[j]}$, for increasing values of $j$. 

\begin{algorithm} [ht]
\caption{Computing the $F$-orbit representative of $A$}
\label{algm: semicanon}
\textbf{Input: } $A \in \mathrm{M}(s,k,p)$ in standard form. \\
\textbf{Output: } $F$-orbit representative of $A$. 
\begin{algorithmic}[1]
\State Compute the partitions $\mathcal{Q}_j$ for $1 \leq j \leq k$ from \Cref{feulner's supp partn}  
\For {$j \in [s, s+1, \ldots, k-1 ]$ and $i \in [s, s-1, \ldots, 1]$}  \Comment{$O(ks)$ time}
        \State $Q,a \gets [i]_{\mathcal{Q}_j}, A_{i,j}$ \Comment{$O(s)$ time}
        \If{$a\neq 0$ \textbf{and} $A_{u,j+1}\neq 0$ for all $u \in Q$ }  
            \Forr {$r \in Q$} 
                Multiply row $A_{r,*}$ by $a^{-1}$ 
            \EndForr \Comment{$O(sk)$ multiplications}
            \ForAll {$l \in [1,2, \ldots, j]$} \Comment{$O(k)$ time}
                \IIf {$\exists q \in Q$ s.t. $A_{q,l} \neq 0$ } 
                    Multiply column $A_{*,l}$ by $a$ 
                \EndIIf \Comment{$O(s)$ time}
            \EndFor
        \EndIf
\EndFor    
\State \Return $A$ 
\end{algorithmic}
\end{algorithm}

\begin{theorem} \label{semicanon is poly} \label{semicanon partition in poly time}
Let $A$ be a matrix in $\mathrm{M}(s,k,p)$ in standard form. 
Then, assuming constant time field operations, \Cref{algm: semicanon} returns the representative of $A^F$ in time polynomial in $k$. 
\end{theorem}

\begin{proof}
\Cref{algm: semicanon} and \cite[Algorithm~1]{feulner} are identical, except: 
Feulner considers codes over arbitrary finite fields, but we restrict to fields of prime order; 
Feulner's algorithm performs row reduction, computes the partitions $\mathcal{Q}_{j}$ and computes the $F$-orbit representative simultaneously, but we assume $A$ is in standard form, and first compute all the $\mathcal{Q}_{j}$. 
Feulner's algorithm also computes the orbit representative under $\mathrm{GL}_j(p) \times D^{[j]}$ for $j < s$. We observe that since $A$ is in standard form, $A_{*, \overline{s}}$ is an identity matrix, so $A$ is the orbit representative of $A$ under the action of $\mathrm{GL}_s(p) \times D^{[s]}$. 
Hence the correctness of \Cref{algm: semicanon} follows from the analysis in \cite{feulner}. 

For the complexity claim, since $s \leq k$ and the analysis of Lines 2--10 of \Cref{algm: semicanon} is straightforward, it remains only to show that the $\mathcal{Q}_j$ can be computed in time polynomial in $sk \in O(k^2)$. 
As $A_{\overline{s},\overline{s}}$ is an identity matrix, if $j \leq s$ then $\mathcal{Q}_j$ is the trivial partition of $\{ 1,2, \ldots, s\}$ with 
$s$ cells. 
Now let $j \geq s+1$ and suppose that we have constructed $\mathcal{Q}_{j-1}$. To construct $\mathcal{Q}_j$, we merge all cells of $\mathcal{Q}_{j-1}$ that have non-trivial intersection with $\Supp(A_{*,j}')$. 
Since $\mathcal{Q}_{j-1}$ has at most $s$ cells, we can do this in time polynomial in $sk$.
\end{proof}

Lastly, we see how we use \Cref{algm: semicanon} to decide if there exists an element of $N_{\Sy_n}(H)$ which induces a given permutation of the $H$-orbits. 

\begin{proposition} \label{suffices to search in Sn}
Let $H$ be a subgroup of $\Sy_n$ in $\InP(\Cy_p)$, let $L= B \rtimes K$ be as in \Cref{normInWreath},  and let $\kappa \in K$. 
Then in time polynomial in $n$, we can determine if there exists $b \in B$ such that $b \kappa \in N_{\Sy_n}(H)$, and if so, output $b$. 
\end{proposition}

\begin{proof}
By \Cref{gi and phii in poly time}, in polynomial time, we can verify that $H \in \InP(\Cy_p)$ and construct generator matrices $M, M' \in \mathrm{M}(s, k, p)$ for $\gamma(H)$ and $\gamma(H^{\kappa^{-1}})$ respectively, in standard form. 
By \Cref{$w$ gives norm iff same semicanon}, such a $b$ exists if and only if $M$ and $M'$ have the same $F$-orbit representative. So by \Cref{semicanon is poly}, in time polynomial in $n/p$, we can decide if such a $b$ exists. Furthermore, 
by keeping track of the changes in Lines 5 and 7 of \Cref{algm: semicanon}, we can simultaneously obtain elements $(R_1, d_1)$ and $(R_2, d_2)$ of $F$ which map $M$ and $M'$ to their $F$-orbit representatives. 
Let $\zeta: B \rightarrow D$ be as in \Cref{conjActionToActionInVector}.
Then $H^{\zeta^{-1}(d_1d_2^{-1})} = H^{\kappa^{-1}}$ and hence we can construct $b = \zeta^{-1}(d_1d_2^{-1})$, in time polynomial in $n$ by \Cref{conjActionToActionInVector}.\ref{can compute im and preim of Xi}.  
\end{proof}


\subsection{Computing \texorpdfstring{$N_{\Sy_n}(H)$}{NSn(H)} by searching in \texorpdfstring{$K$}{K}}
\label{subsection: normFixingOrbs}

In this subsection, we will first show that we can efficiently compute $N_{B}(H)$, then we prove that $N_{\Sy_n}(H)$ can be computed in time $O((\frac{n}{p})! + n^c)$ for some constant $c$. 

\begin{definition} \label{notation: M and X}
Let $M \in \mathrm{M}(s,k, p)$ be the generator matrix of $\gamma(H)$ in standard form. 
For $1 \leq i \leq {s}$, we call $x_i \coloneqq  \gamma^{-1}(M_{i,*})$ the \emph{standard generators of $H$}. We replace the original generating set of $H$ by the \emph{standard generating set} $X = \{x_1, x_2, \ldots, x_s \}$. 
\end{definition}

We now show that each element of $N_B(H)$ conjugates elements of $X$ with non-disjoint supports to the same power. Recall that we identify $\mathds{F}_p$ with the integers $\{0,1, \ldots, p-1\}$.

\begin{lemma} \label{mathcalN sends gens to same power if supp intersects}
Let $b \in N_{B}(H)$, and let $x$ and $y$ be standard generators for $H$ with non-disjoint supports. 
Then there exists an $a \in \mathds{F}_p^*$ such that $x^{b} = x^a$ and $y^{b} = y^a$. 
\end{lemma}

\begin{proof}
As $x^{b} \in H$, the image $\gamma(x^{b})$ is a linear combination of the rows of $M$.
Since $M_{*, \overline{s}}$ is an identity matrix and each element of $\langle M \rangle$ is completely defined by its first $s$ coordinates, 
there exists an $a \in \mathds{F}_p^*$ such that $\gamma(x^{b}) = \gamma(x)a$ and so $x^{b} = x^{a}$.
Similarly, there exists an $a'$ such that $y^{b} = y^{a'}$.  
Now consider an $H$-orbit $\Omega_i$ contained in $\Supp(x) \cap \Supp(y)$ and let $g_i$ be the corresponding generator of the enveloping group $G$.  
Since $b$ fixes each $\Omega_i$ setwise, $g_i^{b} = g_i^a = g_i^{a'}$ and so $a=a'$. 
\end{proof}

For $Q \leq \Sy_n$, a direct product decomposition $Q = \factor_1 \times \factor_2 \times \cdots \times \factor_r$ of $Q$ is a \emph{finest disjoint direct product decomposition} of $Q$ if the groups $\factor_i$ have pairwise disjoint supports, and each $\factor_i$ cannot be written as a non-trivial direct product of subgroups with disjoint supports. 
We use such decompositions to compute the subgroup of $N_{\Sy_n}(H)$ which setwise stabilises each $H$-orbit. 

\begin{proposition} \label{computing norm in inner action}
Let $H$ be a subgroup of $\Sy_n$ in class $\InP(\Cy_p)$ with orbits $\Omega_1, \Omega_2, \ldots, \Omega_k$, and let $B = \langle N_{\Sym(\Omega_i)}(H|_{\Omega_i}) \mid 1\leq i \leq k \rangle$. 
Then $N_{B}(H)$ can be computed in polynomial time.
\end{proposition}

\begin{proof}
Let $t$ be a primitive element of $\mathds{F}_p^*$, and 
let $H = \factor_1 \times \factor_2 \times \cdots \times \factor_r$ be the finest disjoint direct product decomposition of $H$. 
For $1 \leq i \leq r$, let $\Gamma_i = \Supp(R_i)$ and let $\sigma_i \in \Sym(\Gamma_i)$ be such that $g_j^{\sigma_i} = g_j^{t}$ for all generators $g_j$ of $G$ such that $\Omega_j \subseteq \Gamma_i$. We first show that
\begin{equation*} \label{eqn: NBH}
    N_{B}(H) = \langle G, \sigma_1, \sigma_2, \ldots , \sigma_r \rangle. 
\end{equation*}
$\geq$: 
Since $B$ fixes each $H$-orbit setwise, $C_B(H) \leq C_{\Sym(\Omega_1)}(H|_{\Omega_1}) \times \cdots \times C_{\Sym(\Omega_k)}(H|_{\Omega_k}) = G$, and so $G = C_B(H)\leq N_B(H)$. 

Let $h \in H$. Then $h^{\sigma_1} = h_1^{\sigma_1} h_2^{\sigma_1} \ldots h_r^{\sigma_1}$ with $h_j \in R_j$. Fix $1 \leq j \leq r$.
If $\Omega_i \subseteq \Gamma_j$ then $h_j|_{\Omega_i} = g_i^{a}$ for some $a$, and so \[
(h_j^{\sigma_j})|_{\Omega_i} = (h_j|_{\Omega_i})^{\sigma_j} = (g_i^{a})^{\sigma_j} = g_i^{a t} = (h_j|_{\Omega_i})^t.
\]
Since $t$ is independent of $\Omega_i$, the permutation $h_j^{\sigma_j} = h_j^t$ is in $R_j$. 
For $i \neq j$, since $h_i$ and $\sigma_j$ have disjoint supports, $h_i^{\sigma_j} = h_i$. 
Hence $h^{\sigma_j} \in H$ and so $\sigma_j \in N_B(H)$.  \\
$\leq$: 
Let $b \in N_B(H)$ and let $X$ be the standard generating set of $H$. 
By \Cref{mathcalN sends gens to same power if supp intersects}, there exists $t^l \in \mathds{F}_p^*$ such that $x^{b|_{\Gamma_i}} = x^{t^l}$ for all $x \in X$ with support intersecting non-trivially with $\Gamma_i$. 
So $(x|_{\Gamma_i})^{b|_{\Gamma_i}} = (x|_{\Gamma_i})^{t^l} = (x|_{\Gamma_i})^{\sigma_i^l}$. 
Since $R_i = \langle (x|_{\Gamma_i}) \mid x \in X \rangle$, the permutation $b|_{\Gamma_i}$ is in $\langle C_{B}(R_i) , \sigma_i^l \rangle \leq \langle G, \sigma_i \rangle$, so the result follows.  \medskip \\
For the complexity result, in time polynomial in $p$, we can determine a primitive $t \in \mathds{F}_p^*$. In time polynomial in $n$, we can compute the unique finest disjoint direct decomposition of $H$ by \cite{DDPD} and construct the $g_i$ by \Cref{gi and phii in poly time}. 
The $\Sym(\Omega_i)$-conjugacy of permutations can be solved in polynomial time, so we can construct the $\sigma_i$ in polynomial time. 
\end{proof}


Lastly, we show that to compute $N_{\Sy_n}(H)$ for $H \in \mathfrak{InP}(\Cy_p)$, it suffices to search in $K$. Our implementation to compute $N_{\Sy_n}(H)$ which will be described in \Cref{section: algorithm} uses the procedure described in the following proof.

\begin{theorem} \label{thm: norm sym for our H in expo klogk}
\textsc{Norm-Sym} for $H = \langle X \rangle \leq \Sy_n$, where $H \in \InP(\Cy_p)$, can be solved in time $2^{O(\frac{n}{p} \log{\frac{n}{p}} + \log{n})}$. 
\end{theorem}

\begin{proof}
By \Cref{detect H in InP(Cp) in poly time}, in polynomial time $2^{O(\log{n})}$, we can recognise that $H \in \InP(\Cy_p)$, and obtain a generating set $\langle \overline{\phi_i} \mid 2 \leq j \leq k \rangle$ for $K$, an isomorphism $\gamma: G \rightarrow \mathds{F}_p^k$ as in \Cref{def: mapping direct prod of cyclic group to field}, and 
a generator matrix $M$ for $\gamma(H)$ in standard form. 
Next we compute $C_{\Sy_n}(H)$ and $N_{B}(H)$, in time polynomial in $n$ by \Cref{elem polytime}.\ref{cent in sym} and \Cref{computing norm in inner action} respectively. 

Initialise $N$ as $ \langle C_{\Sy_n}(H), N_{B}(H) \rangle$. 
For each $\kappa \in K$, we determine if there exists $b \in B$ such that $b \kappa \in N_{\Sy_n}(H)$, in time polynomial in $n$ by \Cref{suffices to search in Sn}. If such a $b$ exists, we update $N$ as $\langle N, b \kappa \rangle$.
This takes time $O((\frac{n}{p})! n^c) = 2^{O(\frac{n}{p} \log{\frac{n}{p}} + \log{n})}$, for some constant $c$. 

By the end of the procedure, we have $N \leq N_{\Sy_n}(H)$. 
To show $N_{\Sy_n}(H) \leq N$, let $\nu \in N_{\Sy_n}(H)$.
Then since $\nu \in L =B \rtimes K$ by \Cref{normInWreath}.\ref{norm of H in wreath}, we can find $\kappa \in K$ and $b \in B$ such that $\nu = b \kappa$ by \Cref{normInWreath}.\ref{norm to BK}. If $\kappa=1$, then $\nu \in N_B(H)$ and so $\nu \in N$. 
Suppose now that $\kappa \neq 1$. 
We have already found an element $b' \in B$ such that $b' \kappa \in N_{\Sy_n}(H)$, so $b' \kappa \in N$. 
Then $b \kappa (b' \kappa)^{-1} = b b'^{-1} \in N_B(H) \leq N$. Therefore $\nu = b \kappa  \in N$. 
\end{proof}

\Cref{main theorem - cpk} now follows. 
Note that \Cref{limit depth by base size} gives better complexity than \Cref{thm: norm sym for our H in expo klogk} when $p \leq k$. 
However, the algorithm for \Cref{limit depth by base size} requires the checking of all elements of $(\mathds{F}_p^*)^m$, which we do not have useful pruning methods for. When $p$ and $m$ are large, this becomes infeasible (see \Cref{fig:limitDepthvsReduceBranchings}).


\section{Pruning techniques}
\label{subsection: pruning}

Recall that we compute $N_{\Sy_n}(H)$ using backtrack search. 
In this section, we present some methods to compute with $H$ efficiently using $\gamma(H)$, including some pruning techniques.   
We will see how to apply these results in \Cref{section: algorithm}. 
Throughout, let $H$ be as in \Cref{cpk hypo}.

First we show that stabilisers and the relation $\equiv_H$ from \Cref{defn: equiv orbs} can be computed using a generating matrix $M$ for $\gamma(H)$. 
Recall that we denote the rows and columns of $M$ by $M_{i,*}$ and $M_{*,j}$, and let $\overline{s} = \{ 1,2, \ldots, s\}$. 

\begin{lemma} \label{get stab from mat}
Let $\beta$ be a point in an $H$-orbit $\Omega_j$.
If $M_{*,j}$ is non-zero in a unique row, $i$ say, 
then $\gamma(H_{(\beta)}) = \langle M_{\overline{s}\backslash {i},*} \rangle$.
\end{lemma}

\begin{proof}
Since $M_{\overline{s}\backslash {i},j}$ is the zero vector, $\gamma^{-1}( M_{\overline{s}\backslash {i},*})$ fixes $\beta$, so $\gamma(H_{(\beta)}) \geq \langle M_{\overline{s}\backslash {i},*} \rangle$. 
For the other containment, let $h \in H_{(\beta)}$. 
Since $H|_{\Omega_j}$ is regular, $(H|_{\Omega_j})_{(\beta)}=1$, so $\gamma(h)_j = 0$.
Since $M_{i,*}$ is the only row of $M$ with non-zero entry in the $j$-th position, 
$\gamma(h) \in \langle M_{\overline{s} \backslash { i },*} \rangle$.
\end{proof}

The assumption on $M_{*,j}$ in the previous lemma can be achieved by performing row operations. Hence any point stabiliser can be computed by row operations. 

\begin{lemma} \label{id cols are equiv orbits}
Let $M$ be a generator matrix of $\gamma(H)$, and let $a_i$ and $a_j$ be the first non-zero entries of $M_{*,i}$ and $M_{*,j}$ respectively. 
Then $\Omega_i\equiv_H \Omega_j$ if and only if $a_i^{-1} M_{*,i}= a_j^{-1} M_{*,j}$. 
\end{lemma}

\begin{proof}
We show that $\Omega_i \equiv_H \Omega_j$ if and only if $M_{*,j} = aM_{*,i}$ for some $a \in \mathds{F}_p^*$, from which the result follows.   \\
$\Leftarrow$: 
Recall the $p$-cycles $g_i$ from \Cref{defining H as subdir of  Cpk}.
Fix $\alpha \in \Omega_i$ and $\beta \in \Omega_j$.
Define a mapping $\psi : \Omega_i \rightarrow \Omega_j$ by 
setting 
$\psi(\alpha^{g_i^u}) = \beta^{g_j^{au}}$ for $0 \leq u \leq {p-1}$.
Since $\langle g_i \rangle$ and $ \langle g_j^a \rangle$ are isomorphic and regular, $\psi$ is a bijection, and the reader may check that $\psi$ witnesses the equivalence of $\Omega_i$ and $\Omega_j$.\\
$\Rightarrow$: 
By \Cref{defn: equiv orbs}, there exists an involution $\sigma \in \Sym(\Omega_i \cup \Omega_j)$ such that for all $h \in H$, we have $h|_{\Omega_j} = (h|_{\Omega_i})^{\sigma}$.
Let $a$ be such that $g_i^\sigma = g_j^a$. 
Let $h \in H$, and let $v = \gamma(h)_i$ and $u=\gamma(h)_j$. 
Then 
\[(g_j^a)^{v} = (g_i^\sigma)^{v} = (g_i^{v})^\sigma= (h|_{\Omega_i})^\sigma = h|_{\Omega_j} = g_j^{u}, \quad \text{so $u = av$. } \qedhere \] 
\end{proof}


Next, we show how we can use linearly dependent columns of $M$ to prune the search tree.
For a subset $J$ of $\{1,2, \ldots, k\}$, we denote the union $\cup_{j \in J} \Omega_j$ by $\Omega_J$. 
A set $V$ of linearly dependent vectors is \emph{minimally} linearly dependent if no proper subset $U \subset V$ is linearly dependent. 

\begin{lemma} \label{norm maps lin dep to lin dep} \label{indep or mindep by isom}\label{min lin dep to min lin dep} 
Let $I$ and $J$ be subsets of $\{1,2, \ldots, k\}$ such that there exists $\nu \in N_{\Sy_n}(H)$ with $\Omega_I^\nu = \Omega_J$. 
Then the rank of $M_{*,I}$ is equal to the rank of $M_{*,J}$. 
Hence the columns of $M_{*,I}$ are minimally linearly dependent if and only if the columns of $M_{*,J}$ are minimally linearly dependent. 
\end{lemma}

\begin{proof}
The assumption that $\nu$ exists implies that  $H|_{\Omega_I}$ and $H|_{\Omega_J}$ are conjugate in $\Sym(\Omega_I \cup \Omega_J)$. Let $r_I$ and $r_J$ be the ranks of $M_{*,I}$ and $M_{*,J}$ respectively. Then $p^{r_I} = |(H|_{\Omega_I})|= |(H|_{\Omega_J})| = p^{r_J}$, 
so $r_I = r_J$. 

Observe that the columns of $M_{*,I}$ are minimally linearly dependent if and only if $M_{*,I}$ has rank $|I|-1$ and each proper subset $S$ of columns has rank $|S|$. Since rank is preserved by conjugation, the final claim follows. 
\end{proof}


The following lemma is elementary but extremely useful for pruning the search tree $T$. 
Two columns $M_{*,i}$ and $M_{*,j}$ of $M$ are \emph{equivalent}, denoted by $M_{*,i} \equiv_M M_{*,j}$, if $\Omega_i \equiv_H \Omega_j$. 
The \emph{weight enumerator} $w(C)$ of a linear code $C \leq \mathds{F}_p^k$ is the polynomial $w(C) = \sum_{i=1}^{k} w_i x^i$, where $w_i$ is the number of codewords of weight $i$.

\begin{lemma} \label{prune using weight}
\begin{enumerate}
    \item \label{prune by conjugacy classes of stab mats}
    Let $\Delta$ and ${\Gamma}$ be subsets of $\Omega$ such that $|H_{(\Delta)}| = |H_{(\Gamma)}| = p^l$, say. 
    Let $M_{(\Delta)} , M_{(\Gamma)} \in \mathrm{M}(l,k,p)$ be generator matrices of $H_{(\Delta)}$ and $H_{(\Gamma)}$ respectively. 
    If there exists $\nu \in N_{\Sy_n}(H)$ such that $\Delta^\nu = \Gamma$, then 
    there exists a bijection between the $(\equiv_{M_{(\Delta)}})$-classes and the $(\equiv_{M_{(\Gamma)}})$-classes that preserves the class sizes. 
    \item \label{conjugates have the same weight enumerator}
    Let $Q$ and $Q'$ be $\Sy_n$-conjugate subgroups of $H$. Then $w(\gamma(Q)) = w(\gamma(Q'))$.
    \item \label{min wt vec is invariant}
    Let $V$ be the set of non-zero minimum weight vectors in $\gamma(H)$. Then $N_{\Sy_n}(H)$ setwise stabilises $\gamma^{-1}(V)$.
    Hence $N_{\Sy_n}(H)$ setwise stabilises the partition of $\{ 1,2, \ldots, k \}$ where $i \sim j$ if and only if $|\{ v \in V \mid v[i] \neq 0\}| = |\{ v \in V \mid v[j] \neq 0\}|$.  
\end{enumerate}
\end{lemma}

\begin{proof}
For \Cref{prune by conjugacy classes of stab mats}, by \Cref{prune by stabs or projections}.\ref{stabToStab}, the existence of $\nu$ implies that 
$H_{(\Delta)}$ and $H_{(\Gamma)}$ are conjugate in $\Sy_n$, so the required bijection exists by \Cref{conjn maps equiv class to equiv class}. 
For \Cref{conjugates have the same weight enumerator,min wt vec is invariant}, notice that $\gamma(h)_i \neq 0$ if and only if $h|_{\Omega_i} \neq 1$. So the weight of $\gamma(h)$ is the number of $p$-cycles of $h$.
\end{proof}

There is no known polynomial-time algorithm for computing the weight enumerator, so we use a simple heuristic to determine when we use \Cref{conjugates have the same weight enumerator} (see \Cref{algm: compareStabs}).


Lastly, we give a technical lemma that can be used to prune the search tree $T$, whose rationale is as follows. 
Assume we are at a node $\tau$ at depth $m>s$ in $T$. If there exists a $\kappa \in K$ represented by a leaf below $\tau$ such that $b \kappa \in N_{\Sy_n}(H)$, 
then there exists $d \in D$ such that $d \rho(\kappa) \in \MAut(\gamma(H))$.
Let $M' \coloneqq  M{d \rho(\kappa)}$.
Then since $m>s$, at node $\tau$ we know, up to $s$ unknown scalars from $\mathds{F}_p^*$, the first $s$ columns of $M'$. 
Since rows of $M'$ are linear combinations of the rows of $M$ and the elements of $\langle M \rangle$ are defined by their first $s$ coordinates, we now know the whole of $M'$, up to $s$ unknown scalars. So if we can deduce that some entries of $M'$ must be zero, then we can sometimes show that no such $\kappa$ exists, and hence we can backtrack from $\tau$. 

\begin{lemma} \label{deep prune lemma}
Let $M \in \mathrm{M}(s,k, p)$ be a generator matrix of $\gamma(H)$ in standard form. 
Fix $m \in \{s, s+1, \ldots,k-1\}$ and let $J = \{1, 2, \ldots, m\} \dot{\cup} \{u\} \subseteq \{1,2, \ldots, k\}$. 
Let $f: J \rightarrow \{1,2, \ldots, k\} $ be an injection.
If there exists an $i \in \{1,2, \ldots, s\}$ such that 
\begin{equation} \label{deep prune eqn}
M_{i,f(u)} \neq 0 \text{ and } M_{i, f(j) }M_{j,u}=0 \quad \text{ for all } j \in J \cap \{1,2, \ldots, s\},
\end{equation}
then there does not exist $\nu \in N_{\Sy_n}(H)$ such that $\Omega_j^\nu = \Omega_{f(j)}$ for all $j \in J$. 
\end{lemma}

\begin{proof}
Aiming for a contradiction, assume that such an $i$ and $\nu$ exist. 
Then by \Cref{conjActionToActionInVector}.\ref{induces homom of norm and kernel}, $\Xi(\nu) \in \MAut(\gamma(H))$. 
So there exist $d= diag(d_1, d_2, \ldots, d_k) \in D$ and $q \in P$ such that $\Xi(\nu) = dq$. Let $M' = M dq$. Then by considering the action of $W$ on $\mathds{F}_p^k$,  
\begin{equation} \label{eqn: rows of M'}
    M'_{i,j} = d_{f(j)} M_{i, f(j)}, \quad \text{for all $j \in J$ and all $i$.} 
\end{equation}
In particular, $M'_{i,u} = d_{f(u)} M_{i, f(u)} $, which is non-zero. 

Each row of $M'$ is a linear combination of the rows of $M$, and $M$ is in standard form, so 
\[
M'_{i,*} = \sum_{l=1}^s M'_{i,l} M_{l,*}  =  \sum_{l=1}^s d_{f(l)} M_{i,f(l)} M_{l,*} , 
\]
where the second equality follows from (\ref{eqn: rows of M'}). 
Therefore by (\ref{deep prune eqn}), $M'_{i,u} =  \sum_{l=1}^s d_{f(l)}M_{i,f(l)} M_{l,u} =0$, a contradiction.
\end{proof}


\section{Description of implemented algorithm}
\label{section: algorithm}

Given a generating set $X$ of $H \leq \Sy_n$ such that $H \in \InP(\Cy_p)$, we compute $N_{\Sy_n}(H)$ using \Cref{algm: main norm algorithm}.
The algorithm roughly follows the description in the proof of \Cref{thm: norm sym for our H in expo klogk}, with extra steps for the pruning of the search tree. 

We first represent $H$ and $H^{\bot}$ by generator matrices $M$ and $M^{\bot}$ respectively, where $H^{\bot}$ is as in \Cref{defn: dual code}. Then we compute $C_{\Sy_n}(H)$, which may allow us to compute the normaliser of a group with a smaller degree, as in \Cref{can remove equiv orbs}.
Next we compute $N_{B}(H)$, where $B$ is as in \Cref{normInWreath}. That is, we compute all normalising elements fixing every $H$-orbit setwise.
Lastly, we perform backtrack search in $K \cong \Sy_k$ to find the remaining normalising elements. 

To simplify notation, we assume in the description below that $\gamma(H)$ has dimension $s \leq k/2$ and that $H$ has no equivalent orbits. 

\begin{algorithm}
\caption{Computing the normaliser of $H \in \InP(\Cy_p)$}
\label{algm: main norm algorithm}
\textbf{Input:} A generating set $X$ of $H \leq \Sy_n$, where $H \in \InP(\Cy_p)$ and $H$ has no equivalent orbits. \\
\textbf{Output:} $N_{\Sy_n}(H)$  
\begin{algorithmic}[1]
\IIf {$H \not \in \InP(\Cy_p)$} \Return{\textsc{Fail}} \EndIIf \Comment{polynomial by \Cref{detect H in InP(Cp) in poly time}}
\State Compute the enveloping group $G$ of $H$ and $\gamma:G \rightarrow \mathds{F}_p^k$ \Comment{polynomial by \Cref{gi and phii in poly time}}
\State Let $\mathcal{O} = \{ \Omega_1, \Omega_2, \ldots, \Omega_k \}$ be the orbits of $H$, ordered as in \Cref{cpk hypo} 
\State Let $M$ be a generator matrix of $\gamma(H)$ in standard form
\State Let $M^{\bot}$ be a generator matrix of $\gamma(H^{\bot})$ \Comment{where $H^{\bot}$ is as in \Cref{defn: dual code}}
\State $N \gets N_B(H)$ \Comment{using \Cref{computing norm in inner action}}
\State $\textsc{doms} \gets \textsc{domainsInit}(M, M^{\bot})$ \Comment{see \Cref{algm: make partitions}}
\State $LDcols \gets \{ \{i\} \cup \{j \mid M_{j,i} \neq 0,  1 \leq j \leq s \}  \mid s+1 \leq i \leq k \}$ 
\State ${dualLDScols} \gets \{ \{i\} \cup \{j \mid  (M^{\bot})_{j,i}  \neq 0,  k-s+1 \leq j \leq k  \}  \mid 1 \leq  i \leq k-s \}$ 
\State{$\textsc{recurseSearch}([\,], \textsc{doms})$} \Comment{see \Cref{algm: search step}}
\State \Return $N$
\end{algorithmic}
\end{algorithm}


\subsection*{Lines 1--9 of \Cref{algm: main norm algorithm}: Pre-search}
\label{subsection: pre-search}

Before the backtrack search, we compute various structures we shall use later. 

\begin{description} \setlength{\itemsep}{0pt}

\item[Lines 4--5] We compute a row reduced generator matrix $M$ of $\gamma(H)$ using $X$. Since $M$ is in standard form, $M = (\mathrm{I}_s \mid M_0)$, so $M^{\bot} \coloneqq (M_0^{T} \mid \mathrm{I}_{k-s})$ is a generator matrix of $\gamma(H)^{\bot}$ \cite{lintCoding}. 

\item[Line 6] We gather the normalising elements we find in a group $N$. Since we assume that $H$ has no equivalent orbits, $G = C_{\Sy_n}(H) \leq N_B(H)$ by \Cref{cent of subdir}.

\item[Line 7] For each $H$-orbit $\Omega_i$, we compute a set of possible images of $\Omega_i$ under $N_{\Sy_n}(H)$, called the \emph{domain} of $\Omega_i$. This will be used to guide our search in \Cref{algm: search step}.

\item[Lines 8--9] We compute certain sets of linearly dependent columns of $M$ and $M^{\bot}$, and will later explain how we use these sets for pruning. 
Since $M$ is in standard form, the standard basis of $\mathds{F}_p^{s}$ (as column vectors) forms the first $s$ columns. Therefore we can write each later column of $M$ as a linear combination of $M_{*,1}, M_{*,2}, \ldots, M_{*,s}$. 
Similarly, the standard basis of $\mathds{F}_p^{{k-s}}$ forms the last $k-s$ columns of $M^{\bot}$.
\end{description}

By the end of Line 6, $N$ contains $N_B(H) \geq C_{\Sy_n}(H)$.
So, as in \Cref{thm: norm sym for our H in expo klogk}, this allows us to only search for non-trivial elements of $K$ which give rise to elements of $N_{\Sy_n}(H)$.

\subsection*{\Cref{algm: make partitions}: Initialising the domains with \textsc{domainsInit}, \textsc{makeDualEquivPartn}, \textsc{makeStabsPartn} and \textsc{makeInvSetPartn}}
\label{subsection: domain init}

\textsc{DomainsInit} returns a list of sets $\textsc{doms}$, where each entry $\textsc{doms}[i]$ is a subset of $\{1,2 \ldots, k\}$ such that if there exists $\nu \in N_{\Sy_n}(H)$ with $\Omega_i^\nu = \Omega_j$ then $j \in \textsc{doms}[i]$.
We shall compute partitions $\mathcal{P}_e^{\bot}, 
\mathcal{P}_s, \mathcal{P}_s^{\bot}, \mathcal{P}_I$ of $\{1,2, \ldots, k \}$, where each partition represents a partition of $\mathcal{O} = \{ \Omega_1, \Omega_2, \ldots, \Omega_k \}$ preserved by $N_{\Sy_n}(H)$. Hence the meet $\mathcal{P}^*$ of these partitions is also preserved by $N_{\Sy_n}(H)$ and we initialise the domain of $i$ to be $[i]_{\mathcal{P}^*}$. 

We have assumed that the $H$-orbits are pairwise inequivalent, but the relation $\equiv_{H^{\bot}}$ may be non-trivial. \textsc{makeDualEquivPartn} exploits this and computes $\mathcal{P}_e^{\bot}$, which $N_{\Sy_n}(H)$ preserves by \Cref{conjn maps equiv class to equiv class}. 
Since known subgroups of $N_{\Sy_n}(H)$ can be used to prune the search tree, and centralising elements for $H^{\bot}$ that act non-trivially on the set of $H^{\bot}$-orbits yield elements of $N_{\Sy_n}(H)\backslash B$, we also update $N$ with such normalising elements. 

\textsc{makeStabsPartn} considers the stabiliser of each $H$- and $H^{\bot}$-orbit, then uses the sizes of these stabilisers' $\equiv$-classes to make partitions $\mathcal{P}_s$ and $\mathcal{P}_s^{\bot}$, which are preserved by $N_{\Sy_n}(H)$ by Lemmas \ref{prune using weight}.\ref{prune by conjugacy classes of stab mats} and \ref{search in dual}. 

\textsc{makeInvSetPartn} first computes 
the set $invSet$ of minimal weight vectors of $\gamma(H)$. 
The algorithm systematically considers linear combinations of rows of $M$ with increasingly many non-zero coefficients. 
Since $M_{\overline{s},\overline{s}}$ is the identity matrix, 
if $v$ is a linear combination with non-zero coefficients of $i$ rows of $M$, then $wt(v) \geq i$. As we are only interested in vectors with weight at most $m$, we only consider all such linear combinations $v$ of up to $m$ rows of $M$. 
The group $N_{\Sy_n}(H)$ preserves $\mathcal{P}_I$ by \Cref{prune using weight}.\ref{min wt vec is invariant}.

\begin{algorithm}[ht!]
\caption{Initialise domains}
\label{algm: make partitions}
\begin{algorithmic}[1]
\Procedure{domainsInit}{$M, M^{\bot}$}   
    \State $\mathcal{P}_e^{\bot}, 
\mathcal{P}_s, \mathcal{P}_s^{\bot}, \mathcal{P}_I 
\gets  \textsc{makeDualEquivPartn}(M^{\bot}), \textsc{makeStabsPartn}(M), $ \newline \hspace*{10em} $ \textsc{makeStabsPartn}(M^{\bot}), \textsc{makeInvSetPartn}(M)$
    \State $\mathcal{P}^* \gets Meet(\mathcal{P}_e^{\bot},\mathcal{P}_s, \mathcal{P}_s^{\bot}, \mathcal{P}_I)$ 
    \State \Return $[ [i]_{\mathcal{P}^*} \mid 1 \leq i \leq k] $ 
\EndProcedure
\smallskip
\Procedure{makeDualEquivPartn}{$M^{\bot}$}
    \State Compute the $(\equiv_{H^{\bot}})$-classes \Comment{using \Cref{id cols are equiv orbits}}
    \For{each pair of equivalent orbits $\Omega_i$ and $\Omega_j$}
        \State Let $c \in C_{\Sy_n}(H^{\bot})$ conjugate $\Omega_i$ to $\Omega_j$ \Comment{as in \Cref{cent of subdir} }
        \State Find $b \in B$ and $\kappa \in K$ such that $c=b \kappa$ \Comment{using \Cref{normInWreath}.\ref{norm to BK}} 
        \State $N \gets \langle N,b^{-1}\kappa \rangle$ \Comment{$b^{-1}\kappa \in N_{\Sy_n}(H)$ by \Cref{search in dual}}
    \EndFor
    \State \Return{partition of $\{ 1,2, \ldots, k\}$ such that $i \sim j$ if and only if $|[\Omega_i]_{\equiv_{H^{\bot}}}| = |[\Omega_j]_{\equiv_{H^{\bot}}}|$}
\EndProcedure
\smallskip
\Procedure{makeStabsPartn}{$mat$} \Comment{$mat$ is $M$ or $M^{\bot}$}
    \For{$i \in [1,2, \ldots, k]$}  
        \State $Q_i \gets \gamma^{-1}(\langle mat \rangle)_{(\Omega_i)}$ \Comment{using \Cref{get stab from mat}}
        \State $\mathcal{Q}_i \gets $ the multiset of sizes of the $(\equiv_{Q_i})$-classes \Comment{using \Cref{id cols are equiv orbits}}
    \EndFor
    \State \Return{partition of $\{1,2, \ldots, k \}$ such that $i \sim j$ if and only if $\mathcal{Q}_i=\mathcal{Q}_j$}
\EndProcedure   
\smallskip
\Procedure{makeInvSetPartn}{$M$}
    \State $m \gets min_{1 \leq i \leq s}(wt(M_{i,*}))$\Comment{minimum weight of codewords found so far}
    \State $invSet \gets \{i \mid wt(M_{i,*}) = m \}$  \Comment{minimum weight codewords found so far}
    \For{$i \in [2,3,\ldots, m]$}
        \For{all linear combinations $v$ with non-zero coefficients of $i$ rows of $M$}
            \IIf{$wt(v)<m$} Reset $m \gets wt(v)$ and $invSet \gets \{ v \}$ \EndIIf
            \IIf{$wt(v)=m$} Add $v$ to $invSet$ \EndIIf
        \EndFor
    \EndFor 
    \State \Return{partition of $\{1,2, \ldots, k\}$ such that $i \sim j$ if and only if $|\{ v \in invSet \mid v[i] \neq 0\}|$}  \newline \hspace*{5em} $= |\{ v \in invSet \mid v[j] \neq 0\}|$
\EndProcedure
\end{algorithmic}
\end{algorithm}


\subsection*{\Cref{algm: search step}: Search with \textsc{recurseSearch}}
\label{subsection: search}

Now we shall describe the recursive search, which is initialised in Line 10 of \Cref{algm: main norm algorithm}. We shall traverse the search tree depth first, using the domains \textsc{doms} to guide our search.
Note that $N$ is a global variable that stores the group generated by all elements of $N_{\Sy_n}(H)$ we have found so far. The pruning tests used in Lines 10--16 are presented as \Cref{algm: checkLDS,algm: compareStabs}.

\begin{algorithm}[ht!]
\caption{\textsc{recurseSearch}}
\label{algm: search step}
\begin{algorithmic}[1]
\Procedure{recurseSearch}{$ \underline{\alpha} = [\alpha_1, \alpha_2, \ldots, \alpha_d], \textsc{doms}$}
    \If{$d=k$}
        \State $\kappa \gets$ permutation in $K$ such that $\Omega_i^{\kappa} = \Omega_{\alpha_i}$ for all $i$
        \If {there exists $b \in B$ such that $b \kappa \in N_{\Sy_n}(H) $} 
         \Comment{using \Cref{suffices to search in Sn}}
            \State $N \gets \langle N,b \kappa \rangle$
            \State Backtrack to $[\alpha_1, \ldots, \alpha_{j}]$, where $j$ is the largest integer s.t.\ $\alpha_i = i$ for all $i \leq j$  
        \EndIf        
    \Else 
        \IIf{$[\alpha_1, \alpha_2, \ldots, \alpha_{d-1}] = [1,2, \ldots ,d-1]$ and $\alpha_d \neq d$ and $\alpha_d$ is not minimal in \newline\hspace*{5em}  $\alpha_d^{N_{(\alpha_1, \alpha_2, \ldots, \alpha_{d-1})}}$} \Return{} \EndIIf 
        \State $passed1, \textsc{doms} \gets \textsc{checkLDS}(M,LDcols,\underline{\alpha}, \textsc{doms})$ \Comment{see \Cref{algm: checkLDS}}
        \State $passed2, \textsc{doms} \gets \textsc{checkLDS}(M^{\bot},dualLDScols,\underline{\alpha}, \textsc{doms})$ 
        \State $Mstab, MstabIm \gets \gamma(H_{(\Omega_1, \ldots, \Omega_d)}), \gamma(H_{(\Omega_{\alpha_1}, \ldots, \Omega_{\alpha_d})})$ \Comment{using \Cref{get stab from mat}}
         \State $passed3, \textsc{doms} \gets \textsc{compareStabs}(Mstab,MstabIm,\underline{\alpha}, \textsc{doms})$  \Comment{see \Cref{algm: compareStabs}}
        \State $Mstab^{\bot} \gets$ generator matrix of $\gamma((H^{\bot})_{(\Omega_1, \ldots, \Omega_d)})$ \Comment{using \Cref{get stab from mat}}
        \State  $MstabIm^{\bot} \gets$ generator matrix of $ \gamma((H^{\bot})_{(\Omega_{\alpha_1}, \ldots, \Omega_{\alpha_d})})$  
        \State $passed4, \textsc{doms} \gets \textsc{compareStabs}(Mstab^{\bot},MstabIm^{\bot},\underline{\alpha}, \textsc{doms})$ 
        \IIf{$\neg (passed1 \text{ and } passed2 \text{ and } passed3 \text{ and } passed4)$} \Return{} \EndIIf
        \If{$d>s$} 
        \For{$1 \leq i \leq s$ and $t \in [ s+1 \leq t \leq k \mid M_{i,\alpha_u}M_{u, t} =0 \text{ for all } u]$ }
             \State $\textsc{doms}[t] \gets [j \in \textsc{doms}[t] \mid M_{i,j}=0]$ \Comment{using  \Cref{deep prune lemma}}
        \EndFor
        \EndIf 
        \State $\textsc{doms} \gets \textsc{allDiffRefiner}(\textsc{doms})$ 
        \IIf{$\exists i$ such that $\textsc{doms}[i] = \emptyset$} \Return{} \EndIIf
        \Forr{$\alpha_{d+1} \in \textsc{doms}[d+1]$}
           $\textsc{recurseSearch}([\alpha_1, \alpha_2, \ldots, \alpha_{d+1}], \textsc{doms})$
        \EndForr        
    \EndIf
\EndProcedure
\end{algorithmic}
\end{algorithm}

\begin{description}\setlength{\itemsep}{0pt}
\item [Lines 2--7] If $d=k$, then we have arrived at a leaf node of the search tree. 
We only require a generating set of $N_{\Sy_n}(H)$, so we can backtrack to node $[\alpha_1, \alpha_2, \ldots, \alpha_{j-1}]$. See \cite[\S9.1.1]{seress} for details. 

\item [Line 9] By \cite[\S9.1.1]{seress}, it suffices to test only the minimum value of each $N_{(\Omega_1, \Omega_2, \ldots, \Omega_{d-1})}$-orbit. 

\item [Line 23] 
Since the images of $\mathcal{O}$ must be distinct, whenever the domains are changed, we can  further refine the domains. For each $i$, let $J_i  \coloneqq  \{ j\mid 1 \leq j \leq {k} \text{ and } \textsc{doms}[j] = \textsc{doms}[i] \} $.
If $|J_i| = |\textsc{doms}[i]|$, then any normalising element under the current node maps $\{\Omega_i \mid i \in J_i \}$ to $\{\Omega_i \mid i \in \textsc{doms}[i] \}$. 
We remove elements of $\textsc{doms}[i]$ from $\textsc{doms}[t]$ for all $t \not \in J_i$. 

\item [Lines 24--25] If a domain becomes empty, we backtrack. Otherwise, we continue the depth-first search, branching using $\textsc{doms}$. 
\end{description}

\subsection*{\Cref{algm: checkLDS}: Pruning functions \textsc{checkLDS} and \textsc{compareStabs}}

\textsc{checkLDS} uses minimally linearly dependent columns of $M$ to prune the search tree. If the condition in Line 6 is satisfied then there is no normalising element under the current node that sends $\Omega_{I[1]}$ to $\Omega_i$ by \Cref{norm maps lin dep to lin dep}, so we remove $i$ from the domain of $I[1]$.  

\textsc{compareStabs} uses conjugacy of point stabilisers to prune the search tree, as in \Cref{prune by stabs or projections}.\ref{stabToStab}. If any of the conditions in Lines 13, 15 or 18 is satisfied then there are no normalising elements under the current node, by  Lemmas \ref{conjn maps equiv class to equiv class}, \ref{prune using weight}.\ref{prune by conjugacy classes of stab mats} and \ref{prune using weight}.\ref{conjugates have the same weight enumerator} respectively.

\begin{algorithm}[ht!]
\caption{\textsc{checkLDS} and \textsc{compareStabs} }
 \label{algm: checkLDS} \label{algm: compareStabs}
\begin{algorithmic}[1]
\Procedure{checkLDS}{$M,LDcols, [\alpha_1, \alpha_2, \ldots, \alpha_{d}], \textsc{doms}$}
    \For{$lds \in LDcols$} \Comment{see \Cref{algm: main norm algorithm}, Line 8}
        \State $I \gets lds \backslash \{1,2, \ldots, d \} $ \Comment{unassigned column images}
        \If{$|I| = |lds|-1$} \Comment{image of $M_{*, I[1]}$ must be in the span of other columns}
            \For {$i \in \textsc{doms}[I[1]]$}
                \IIf{$M_{*,i} \not \in \langle M_{*,{\alpha_j}} \mid j \in lds\backslash I[1] \rangle$} 
                   Remove $i$ from $\textsc{doms}[I[1]]$ 
                \EndIIf       
            \EndFor
        \EndIf
    \EndFor
    \State \Return{\textsc{True}, $\textsc{doms}$}
\EndProcedure
\smallskip
\Procedure{compareStabs}{$Mstab,MstabIm,[\alpha_1, \alpha_2, \ldots, \alpha_{d}], \textsc{doms}$}
    \IIf{multisets of the sizes of the $(\equiv_{Mstab})$-classes and of the $(\equiv_{MstabIm})$-classes are different \newline \hspace*{2em} } \Return{$\textsc{False}, \textsc{doms}$}   \EndIIf   

    \For{$i \in \{1,2, \ldots ,k \} \backslash \{\alpha_1, \alpha_2, \ldots, \alpha_d\}$ and $j \in \textsc{doms}[i]$} 
        \IIf{$| [\Omega_i]_{\equiv_{Mstab}} | \neq |[\Omega_j]_{\equiv_{MstabIm}}|$} Remove $j$ from $\textsc{doms}[i]$ \EndIIf 
    \EndFor
    \If{$(s-d)*p \leq 45$} \Comment{45 is a heuristic}  
        \IIf{$w(Mstab) \neq w(MstabIm)$} 
        \Return{$\textsc{False}, \textsc{doms}$} 
        \EndIIf 
    \EndIf
    \State \Return{$\textsc{True}, \textsc{doms}$}   \Comment{no obstruction to conjugacy found}
\EndProcedure
\end{algorithmic}
\end{algorithm}


\section{Extension: Groups in class \texorpdfstring{$\InP(\Di_{2p})$}{InP(D2p)}}
\label{section: dihedral}\label{subdir of dihedral}

In this section, we consider $H \in \mathfrak{InP}(\Di_{2p})$, where $p$ is an odd prime and $\Di_{2p}$ is the dihedral group of order $2p$ and degree $p$. 
We show that $N_{\Sy_n}(H)$ can be found by computing the normalisers of its Sylow subgroups, which can be identified with groups in classes $\InP(\Cy_p)$ and $\InP(\Cy_2)$.

We will assume the following throughout this section. 
Let $n=pk$ and $\Omega= \{1,2, \ldots,  n \}$.
Let $H$ be a subgroup of $\Sy_n$ in class $\InP(\Di_{2p})$ with orbits $\Omega_1, \Omega_2, \ldots, \Omega_k$ such that $\Omega= \cup_{i=1}^k \Omega_i$. 
Let $\Di_i \coloneqq H|_{\Omega_i}$ for each $1 \leq i \leq k$, so $D = D_1 \times D_2 \times \cdots \times D_k \leq \Sy_n$ is the enveloping group of $H$. 
For $1 \leq i \leq {k}$, let $G_i$ be the Sylow $p$-subgroup of $D_i$, let $G = G_1 \times G_2 \times \cdots \times G_k \leq \Sy_n$, and let $H_q$ be a Sylow $q$-subgroup of $H$, where $q \in \{2,p\}$. 

\begin{lemma} \label{sylow p is intersection + norm inside norm of sylow p} \label{T2 results}
\begin{enumerate}
    \item \label{Tp normal in H} \label{H is TpT2} $H = H_p \rtimes H_2 = (H \cap G) \rtimes H_2$, and so $N_{\Sy_n}(H) \leq N_{\Sy_n}(H_p)$. 
    \item \label{Tp subdir of G}$H_p$ is a subdirect product of $G$, and so $H_p \in \InP(\Cy_p)$.
    \item \label{T2 is stab} \label{T2 is subdir of stab} For all $i$, there exists $\alpha_i \in \Omega_i$ such that $H_2|_{\Omega_i} = (D_i)_{(\alpha_i)}$.  
\end{enumerate}
\end{lemma}

\begin{proof}
\Cref{Tp normal in H}: As $G$ is the unique Sylow $p$-subgroup of $D$, it follows that $H_p \leq H \cap G$, so $H_p = H \cap G$. Since $G \trianglelefteq D$, the group $H_p$ is characteristic in $H$, so the last assertion follows. \\ 
\Cref{Tp subdir of G}: 
Since $H$ is a subdirect product of $D$, for all $i$ there exists $h \in H$ such that $1 \neq h|_{\Omega_i} \in G_i$. 
Then, for all $j$, the restriction $h^2|_{\Omega_j}$ is of order $1$ or $p$.
So $h^2|_{\Omega_j} \in G_j$ for all $j$ and therefore $h^2 \in G \cap H = H_p$. 
Since $h|_{\Omega_i}$ is a $p$-cycle, $h^2|_{\Omega_i} \neq 1$, so $H_p|_{\Omega_i} = G_i$.  \\
\Cref{T2 is stab}: 
For all $i$, there exists an involution $r_i \in D_i$ such that $H_2|_{\Omega_i} \leq \langle r_i \rangle$. 
Since every involution in $D_i$ fixes a point, there exists $\alpha_i \in \Omega_i$ such that $r_i$ fixes $\alpha_i$.  
So $H_2|_{\Omega_i} \leq (D_i)_{(\alpha_i)}$.
If $H_2$ pointwise stabilises $\Omega_i$, then by \Cref{H is TpT2}, 
$H|_{\Omega_i} = (H_p H_2)|_{\Omega_i} =  H_p|_{\Omega_i} = G_i$,
but $H$ is a subdirect product of $D$, a contradiction, and so $H_2|_{\Omega_i} = (D_i)_{(\alpha_i)}$. 
\end{proof}

For $2 \leq i \leq k$, let $\phi_i : \Omega_1 \rightarrow \Omega_i$ witness the permutation isomorphism from $D_1$ to $D_i$, and satisfy $\phi_i(\alpha_1)= \alpha_i$, where $\alpha_i$ is as in \Cref{T2 results}.\ref{T2 is stab}. 
Let $K = \langle \overline{\phi_i} \mid 2 \leq i \leq k \rangle$, let $L = \langle N_{\Sym(\Omega_1)}(G_1) ,\ldots, N_{\Sym(\Omega_k)}(G_k), K \rangle \cong (\Cy_p \rtimes \Cy_{p-1}) \wr \Sy_k$ and let $I = N_L(H_p) \cap N_L(H_2)$.

\begin{proposition} \label{frattini on subdir od D2pk}
$N_{\Sy_n}(H) = IH$.
\end{proposition}

\begin{proof}
By \Cref{T2 results}.\ref{H is TpT2}, for all $h \in H$, there exist $h_p \in H_p$ and $h_2 \in H_2$ such that $h=h_ph_2$. 
So $h^\iota = h_p^{\iota} h_2^{\iota} \in H_p H_2 = H$ for all $\iota \in I$. 
Therefore $IH \leq N_{\Sy_n}(H)$. 

For the converse containment, $N_{\Sy_n}(H) = N_L(H)$ by Lemmas \ref{T2 results}.\ref{Tp normal in H} and \ref{define and analyse L}.\ref{norm of H in wreath}. Using the Frattini argument, $N_{\Sy_n}(H)= N_{N_L(H)}(H_2)H = (N_L(H) \cap N_L(H_2))H$.  
Finally by \Cref{sylow p is intersection + norm inside norm of sylow p}.\ref{Tp normal in H}, $N_L(H)$ is contained in $N_L(H_p)$, so $N_{\Sy_n}(H) \leq IH$. 
\end{proof}

For all $i \leq k$, let $N_i = N_{\Sym(\Omega_i)}(H_2|_{\Omega_i}) \cap N_{\Sym(\Omega_i)}(G_i)$, and let $T = \langle N_1, N_2, \ldots, N_k, K \rangle$. 

\begin{lemma}\label{get N of T2} 
\begin{enumerate}
    \item \label{I in T} $I \leq T$ and so $I = N_T(H_2) \cap N_T(H_p)$.
    \item \label{intersection of norms is generated by ci } Let $t \in \mathds{F}_p^*$ be primitive. For all $i \leq k$, let $g_i \in \Sym(\Omega_i)$ be a generator of $G_i$, and let $c_i$ be an element of $\Sym(\Omega_i \backslash \{\alpha_i\})$ such that $g_i^{c_i} = g_i^t$. 
    Then $N_i = \langle c_i \rangle$ and 
    so $T \cong \Cy_{p-1} \wr \Sy_k$.
\end{enumerate}

\end{lemma}
\begin{proof}
\Cref{I in T}: 
Let $\iota \in I$. Since $L \cong N_{\Sym(\Omega_1)}(G_1) \wr \Sy_k$ by \Cref{define and analyse L}.\ref{item: L is wreath}, there exists $\kappa \in K$ such that $\iota \kappa^{-1}$ fixes each $\Omega_i$ setwise and $(\iota \kappa^{-1})|_{\Omega_i} \in N_{\Sym(\Omega_i)}(G_i)$. 
To show that $(\iota \kappa^{-1})|_{\Omega_i}$ normalises $H_2|_{\Omega_i}$, 
let $j$ be such that $\Omega_j = \Omega_i^{\iota}$. 
Then $\kappa$ witnesses the permutation isomorphism from $D_i$ to $D_j$ and maps $\alpha_i$ to $\alpha_j$, so $\kappa$ conjugates $H_2|_{\Omega_i}= (D_i)_{(\alpha_i)}$ to $H_2|_{\Omega_j}$. 
Therefore $(H_2|_{\Omega_i})^{\iota \kappa^{-1}} = (H_2|_{\Omega_j})^{\kappa^{-1}} = H_2|_{\Omega_i}$. 
Hence $(\iota \kappa^{-1})|_{\Omega_i} \leq N_i$ and so $\iota \in T$.\\
\Cref{intersection of norms is generated by ci }: 
Observe that $N_{\Sym(\Omega_i)}(G_i) = \langle g_i, c_i \rangle$, so $N_i
 \leq \langle g_i\rangle \rtimes \langle  c_i \rangle $.
Assume for a contradiction that there exists a non-trivial $g \in \langle g_i \rangle$ and a $c \in \langle c_i \rangle$
such that $gc \in N_i$. Letting $r_i$ be the generator for $H_2|_{\Omega_i} \cong \Cy_2$, note that $r_i^{gc} = r_i$. Since $g$ moves $\alpha_i$ and $c$ fixes only $\alpha_i$, the points $\alpha_i^{gc}$ and $ \alpha_i$ are distinct. 
But $(\alpha_i^{gc})^{r_i} = (\alpha_i^{gc})^{(r_i^{gc})} = \alpha_i^{r_igc} = \alpha_i^{gc}$, a contradiction since $r_i$ only fixes $\alpha_i$, so $N_i \leq \langle c_i \rangle$.

To see that $N_i \geq \langle c_i \rangle$, first notice that $c_i \in N_{\Sym(\Omega_i)}(G_i)$. Since $r_i$ is an element of $N_{\Sym(\Omega_i)}(G_i)$ which fixes $\alpha_i$, it is in $\langle c_i \rangle$, and so $c_i$ normalises $H_2|_{\Omega_i}$. 

Lastly the isomorphism follows from \Cref{define and analyse L}.\ref{item: L is wreath}. 
\end{proof}

Fix a non-trivial orbit $\Omega_{11}$ of $H_2|_{\Omega_1}$. 
For $2 \leq i \leq k$, let $\Omega_{i1}  \coloneqq  \Omega_{11}^{\overline{\phi_i}}$ be an $(H_2|_{\Omega_i})$-orbit, and let $\Gamma = \bigcup_{i=1}^k \Omega_{i1}$. 
Then $H_2|_{\Gamma}$ is in class $\InP(\Cy_2)$, so $N_{\Sym(\Gamma)}(H_2|_{\Gamma})$ can be computed using \Cref{algm: main norm algorithm}.

Next, we show how $N_T(H_2)$ can be constructed from $N_{\Sym(\Gamma)}(H_2|_{\Gamma})$. 
Let $\theta: N_{\Sym(\Gamma)}(H_2|_{\Gamma}) \rightarrow K$ be defined by  $\Omega_{i}^{\theta(g)}= \Omega_{j}$ if $\Omega_{i1}^g= \Omega_{j1}$, and let $R  \coloneqq  \langle N_{\Sym(\Omega_1)}(G_1), \ldots,N_{\Sym(\Omega_k)}(G_k), \mathrm{Im}(\theta) \rangle$. 

\begin{lemma} \label{I in ci wr Sk}
$N_T(H_2) = \langle  N_1, N_2, \ldots, N_k, \mathrm{Im}(\theta) \rangle \leq R$. 
\end{lemma}
\begin{proof}
Observe that as $|(H_2|_{\Omega_i})|=2$, the normaliser and the centraliser of $H_2|_{\Omega_i}$ in $\Sym(\Omega_i)$ coincide, 
so $N_i \leq N_T(H_2)$. 
To show that $\mathrm{Im}(\theta) \leq N_T(H_2)$, let $g \in N_{\Sym(\Gamma)}(H_2|_{\Gamma})$ and let $h \in H_2$.
Then there exists $h' \in H$ such that $(h|_{\Gamma})^{g} = h'|_{\Gamma}$. 
We first show that $ h^{\theta(g)} = h'$. 

Fix $i$ and let $\Omega_{j1}= \Omega_{i1}^g$, so $\Omega_j = \Omega_i^{\theta(g)}$. 
Then $(H|_{\Omega_{i}})^{\theta(g)} =  H|_{\Omega_{j}}$ as $\theta(g) \in K$. 
Since $K$ acts on the fixed points of $H_2$,
\[ 
(H_2|_{\Omega_{i}})^{\theta(g)} = ((H|_{\Omega_{i}} )_{\alpha_i})^{\theta(g)} 
\leq (H|_{\Omega_{j}} )_{\alpha_j} = H_2|_{\Omega_{j}},
\]
and equality is achieved since both groups have order two. 
Now $h|_{\Omega_i} \neq 1$ if and only if $h'|_{\Omega_j}$ and $(h|_{\Omega_{i}})^{\theta(g)}$ are the unique non-trivial element of $H_2|_{\Omega_j}$, and $h|_{\Omega_i} = 1$ if and only if $h'|_{\Omega_j}=(h|_{\Omega_{i}})^{\theta(g)}=1$.  
Therefore $(h^{\theta(g)})|_{\Omega_{j}}  = (h|_{\Omega_{i}})^{\theta(g)} = h'|_{\Omega_j} $, as required.  

So $\theta(g)$ normalises $H_2$, and as $\theta(g) \in K \leq T$, it follows that $\theta(g) \in N_T(H_2)$. 
Therefore $N_T(H_2) \geq \langle  N_1, N_2, \ldots, N_k,  \mathrm{Im}(\theta) \rangle$. 

To show that these two groups are equal, let $\nu \in N_T(H_2)$ and let $\kappa  \coloneqq  \theta(\nu|_{\Gamma}) \in K \leq T$. Then as both $N_T(H_2)$ and $\mathrm{Im}(\theta)$ act on $\mathcal{O}$, it follows that $\Omega_i^{\kappa} = \Omega_j$ if and only if $\Omega_{i1}^{\nu|_{\Gamma}} = \Omega_{j1}$ if and only if $\Omega_{i}^{\nu} = \Omega_{j}$. 
Therefore $\nu \kappa^{-1}$ is an element of $T$ which fixes each $\Omega_i$ setwise. 
By \Cref{get N of T2}.\ref{intersection of norms is generated by ci }, $T \cong N_1 \wr \Sy_k$, so $\nu \kappa^{-1} \in  N_1 \times N_2 \times \cdots \times N_k $. 
\end{proof}

As in \Cref{thm: norm sym for our H in expo klogk}, we may compute $N_{R}(H_p)$ by considering all $\kappa \in \mathrm{Im}(\theta) \leq K$. 
Lastly, we show that $I$ can be computed from $N_{R}(H_p)$.

\begin{lemma}  \label{compute I from norm of groups in inp}
Let $W \leq \mathrm{GL}_k(p)$ be as in \Cref{subsection: norm as aut}, 
let $\Xi: L \rightarrow W$ be as in \Cref{conjActionToActionInVector} and let $\Xi|_{T} : T \rightarrow W$ be the restriction of $\Xi$ to $ T$. 
Then $\Xi|_{T}$ is an isomorphism which maps $I$ to $\Xi(N_{R}(H_p))$. 
\end{lemma}

\begin{proof}
It follows from \Cref{get N of T2}.\ref{intersection of norms is generated by ci } that $L = \langle T, G \rangle$. 
By \Cref{conjActionToActionInVector}.\ref{xi is epi with kernel $G$}, $\Xi$ is an epimorphism and $\Xi(G)=1$, so $W = \Xi(L) = \Xi(T)$, and hence $\Xi|_T$ is an epimorphism. 
Note also that the groups $G \cong \Cy_p^k$ and $T \cong \Cy_{p-1} \wr \Sy_k$ have trivial intersection, so $L/G \cong T$. 
Hence $W \cong L/\Ker(\Xi) \cong T$, therefore $\Xi|_T$ is an isomorphism.

By \Cref{get N of T2}.\ref{I in T} and \Cref{I in ci wr Sk},  
\[ I  = N_T(H_2) \cap N_T(H_p) \leq R \cap N_{T}(H_p) \leq N_{R}(H_p), \]
and so $\Xi(I) \leq \Xi(N_{R}(H_p))$. 
For the other containment, let $r \in N_{R}(H_p)$.
Then $\tau \coloneqq \Xi|_{T}^{-1}(\Xi(r))$ is an element of $T$ normalising $H_p$, so $\tau \in N_T(H_p)$. 

Observe from \Cref{get N of T2}.\ref{intersection of norms is generated by ci } that  $N_{\Sym(\Omega_i)}(G_i) = \langle g_i, N_i \rangle$. So $R = \langle  G, N_1, N_2, \ldots, N_k, \mathrm{Im}(\theta) \rangle$, which is $\langle G, N_T(H_2) \rangle$ by \Cref{I in ci wr Sk}.  
Now as $\Xi(G)=1$ by \Cref{conjActionToActionInVector}.\ref{xi is epi with kernel $G$}, $\Xi(R) = \Xi(N_T(H_2))$. 
This means that $\tau \in N_T(H_2)$ and so $\tau \in I$. Hence $\Xi(r) \in \Xi(I)$ and therefore $\Xi(N_{R}(H_p)) = \Xi(I)$. 
\end{proof}

Therefore, we compute $N_{\Sy_n}(H)$ for $H \leq \Sy_n$ in $\InP(\Di_{2p})$ in the following way. First we compute the Sylow $p$-subgroup $H_p$ and a Sylow $2$-subgroup $H_2$ of $H$. Then we compute $N_{\Sym(\Gamma)}(H_2|_{\Gamma})$ and $N_R(H_p)$
using \Cref{algm: main norm algorithm}. Next we compute $\Xi|_{T}^{-1}(\Xi(N_{R}(H_p)))$, which is equal to $N_L(H_p) \cap N_L(H_2)$ by \Cref{compute I from norm of groups in inp}. 
By \Cref{frattini on subdir od D2pk}, $N_{\Sy_n}(H) = \langle N_L(H_p) \cap N_L(H_2), H \rangle$.


\section{Results}
\label{section: results}
\label{Cpk result}\label{Dpk result}

In our experiments we considered groups $H \leq \Sy_{pk}$ in class $\InP(\Cy_p)$  that are isomorphic to $\Cy_p^{k/2}$, for {$p = 2,3,5,11$.} 
We generate these instances by populating the entries of a $k/2 \times k$ matrix with random elements of $\mathds{F}_p$. 
We rerun the generation if $\mathrm{rank} \, M \neq k/2$.

For each value of $p$ and $k$, we create 10 instances of $H$ and compute  $N_{\Sy_n}(H)$ using both the \texttt{GAP} function \textsc{Normalizer} and our new algorithm, each run with a 10-minute time limit. 
We report the median, lower quartile and upper quartile time, in seconds, in \Cref{fig:Cpk}. 
The lower and upper boundaries of the shaded area give the lower and upper quartiles respectively.

Next we compare the performance of computing $N_{\Sy_n}(H)$ for $H \leq \Sy_n$ in class $\InP(\Cy_p)$ using the methods described in  \Cref{limit depth by base size} and \Cref{thm: norm sym for our H in expo klogk} respectively. 
The results are shown in \Cref{fig:limitDepthvsReduceBranchings}, where \textsc{fullSearch} refers to the algorithm described in \Cref{section: algorithm}. 
To obtain complexity $2^{O(\frac{n}{p} \log{\frac{n}{p} + \log{n})}}$, \textsc{limitDepth} is a combination of methods of \Cref{limit depth by base size} and \Cref{thm: norm sym for our H in expo klogk}. The algorithm is as follows. 
Let $H \in \InP(\Cy_p)$ have order $p^s$. 
As in \Cref{algm: main norm algorithm}, we perform backtrack search in $K$.
At a node at depth $s$, we iterate over all $(p-1)^s$ combination of $(\mathds{F}_p^*)^s$, as in \Cref{limit depth by base size}.  
If we succeed in finding a normalising element $g \in N_{\Sy_n}(H)$, we update $N$ as $\langle N,g \rangle$, else we backtrack.
Results (\Cref{fig:limitDepthvsReduceBranchings}) show that even though \textsc{fullSearch} has a higher worst case complexity, it performs 
much better than \textsc{limitDepth} in practice, especially where $p$ or $s$ are large.
\begin{table}[ht!]
\centering
\begin{minipage}{.45\linewidth}
\begin{tabular}{cccc}
\hline
$p$ & $s$ & \textsc{fullSearch}& \textsc{limitDepth} \\
\hline
5 & 4 &	0.11 &	0.125 \\ 
5 & 6 &	0.4765 & 0.625 \\
5 & 8 &	3.0625 & 2.953 \\
5 & 10 & 8.2415 & 65.75 \\
\hline
\end{tabular} 
\end{minipage}
\quad
\begin{minipage}{.45\linewidth}
\begin{tabular}{cccc}
\hline
$p$ & $s$& \textsc{fullSearch}& \textsc{limitDepth} \\
\hline
2 & 6 & 0.125 &	0.125\\
3 & 6 & 0.2655 & 0.297\\
7 & 6 &	0.9605 & 12.0235 \\
11 & 6 & 5.453 & >>600 \\
\hline
\end{tabular}
\end{minipage}
\caption{Median times (s) to compute $N_{\Sy_n}(H)$ for 10 instances of $H \leq \Sy_{20p}$ in $\InP(\Cy_p)$. }
\label{fig:limitDepthvsReduceBranchings}
\end{table}

We also consider $H \leq \Sy_{pk}$ in class $\InP(\Di_{2p})$ as in \Cref{subdir of dihedral}, for $p = 3,11$.
Using \Cref{T2 results}.\ref{H is TpT2}, we generate $H$ as product of groups in $\InP(\Cy_2)$ and $\InP(\Cy_{p})$. 
We generate these groups that are isomorphic to $\Cy_2^{k/2}$ and $\Cy_p^{k/2}$ respectively using the method described before. 
The rest of the experiment works the same as that for $\InP(\Cy_p)$.
We report computation times in \Cref{fig:dpk}.

All algorithms are implemented in \texttt{GAP}, apart from \textsc{makeInvSetPartn} in \Cref{algm: make partitions}, which is implemented in \texttt{C++}.

\begin{figure}[htp!]
   \centering
   \subfloat[$p=2$]{\includegraphics[width=.45\textwidth,height=.25\textheight]{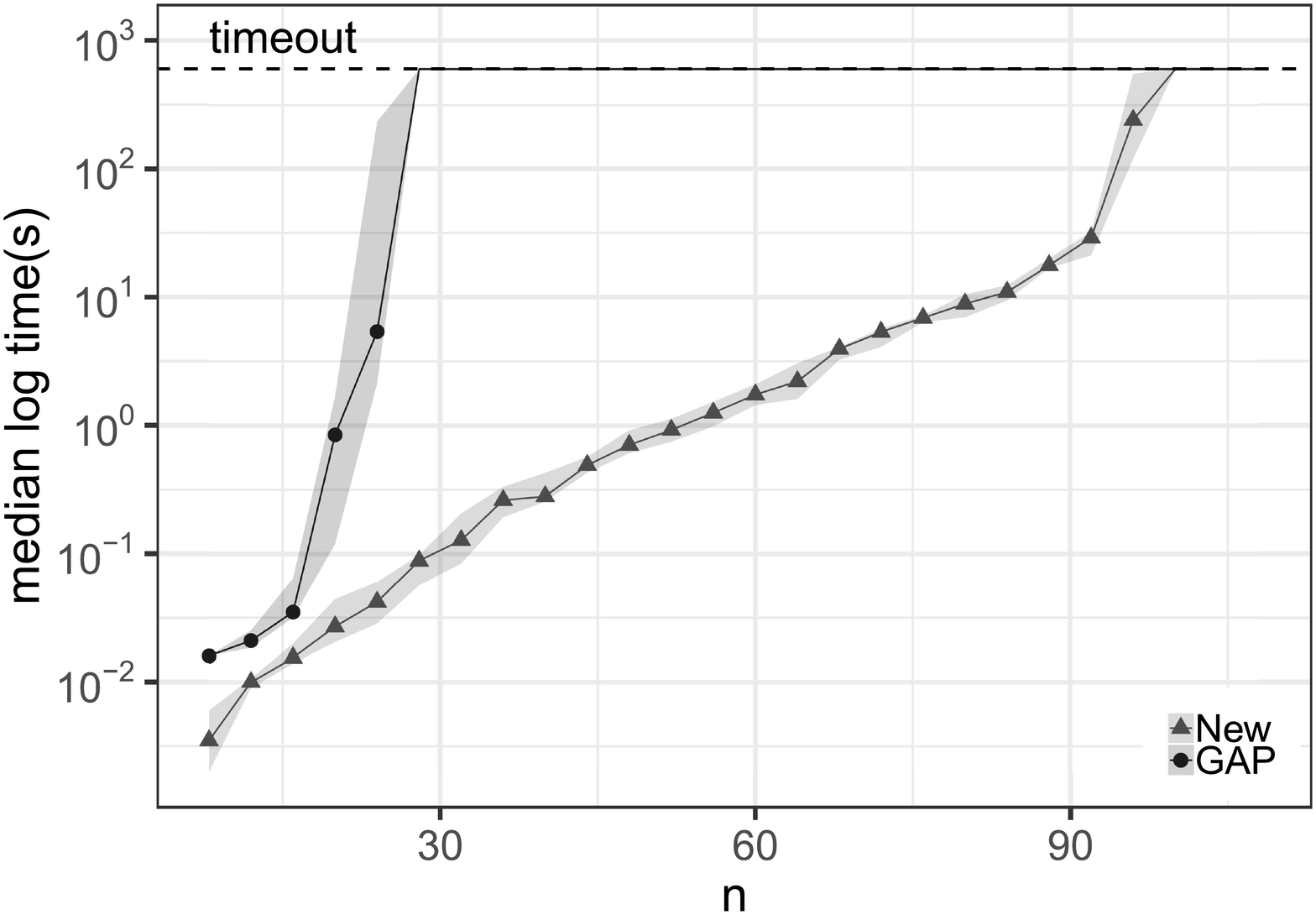}}\quad
   \subfloat[$p=3$]{\includegraphics[width=.45\textwidth,height=.25\textheight]{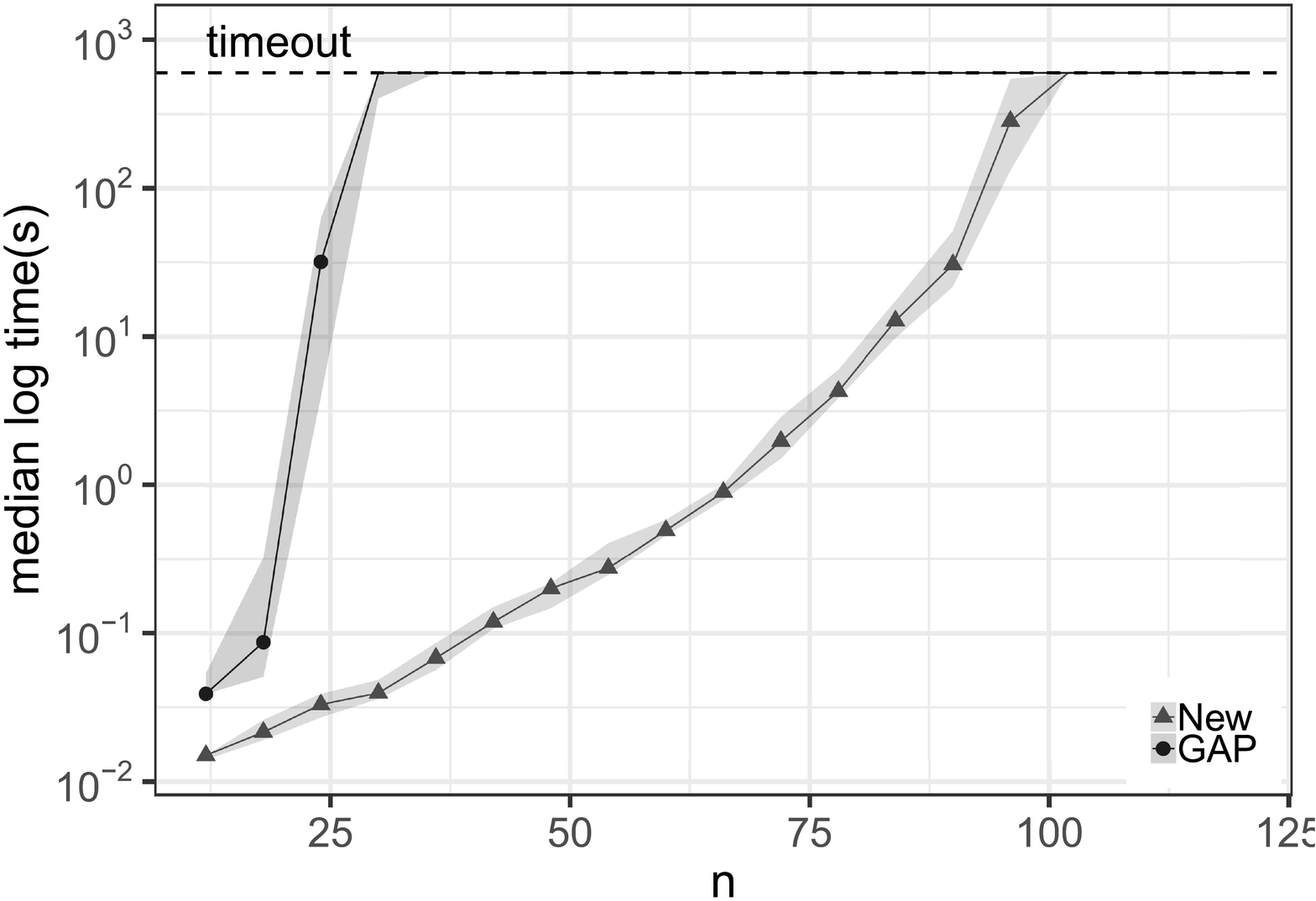}}\\
   \subfloat[$p=5$]{\includegraphics[width=.45\textwidth,height=.25\textheight]{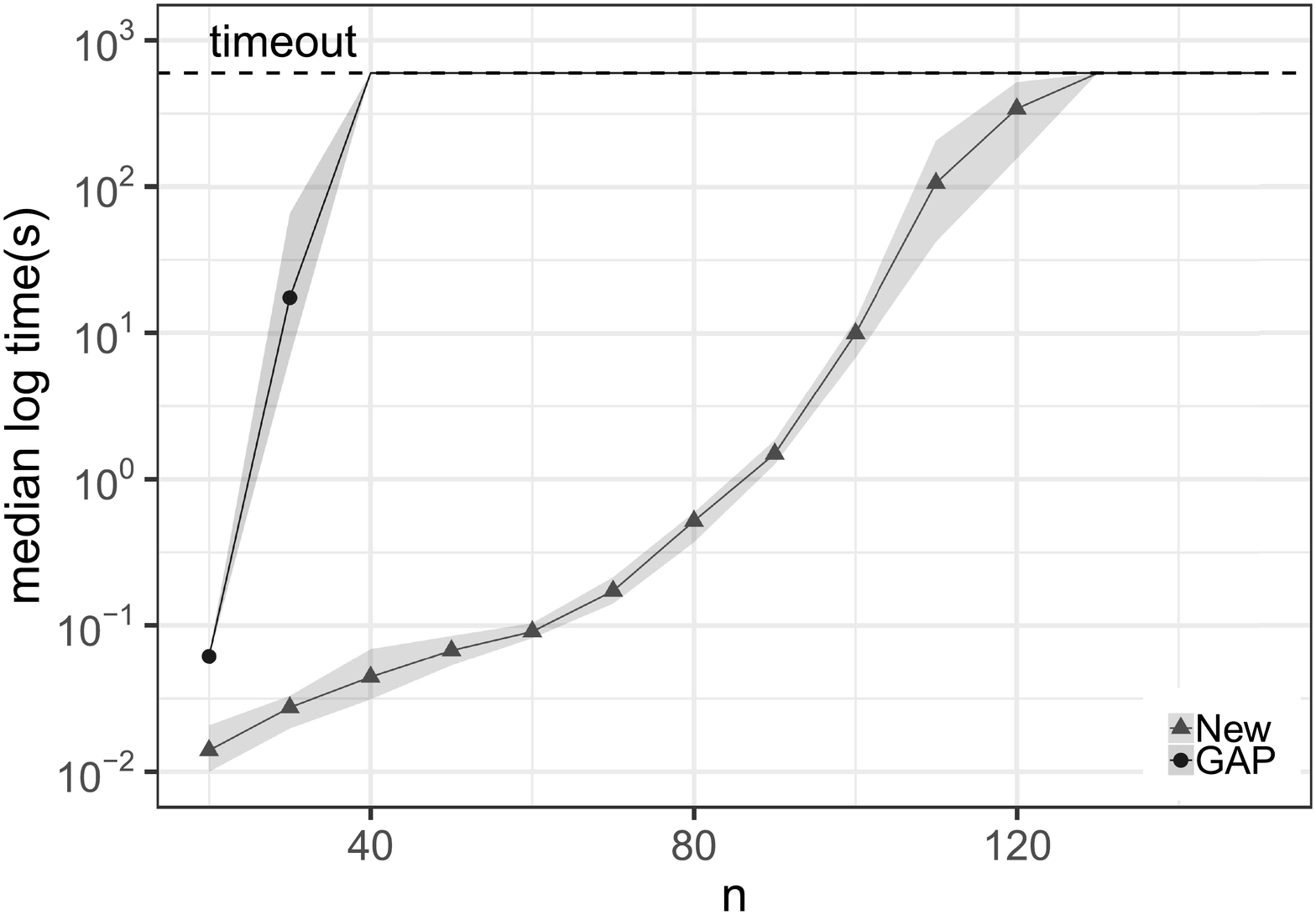}}\quad
    \subfloat[$p=11$]{\includegraphics[width=.45\textwidth,height=.25\textheight]{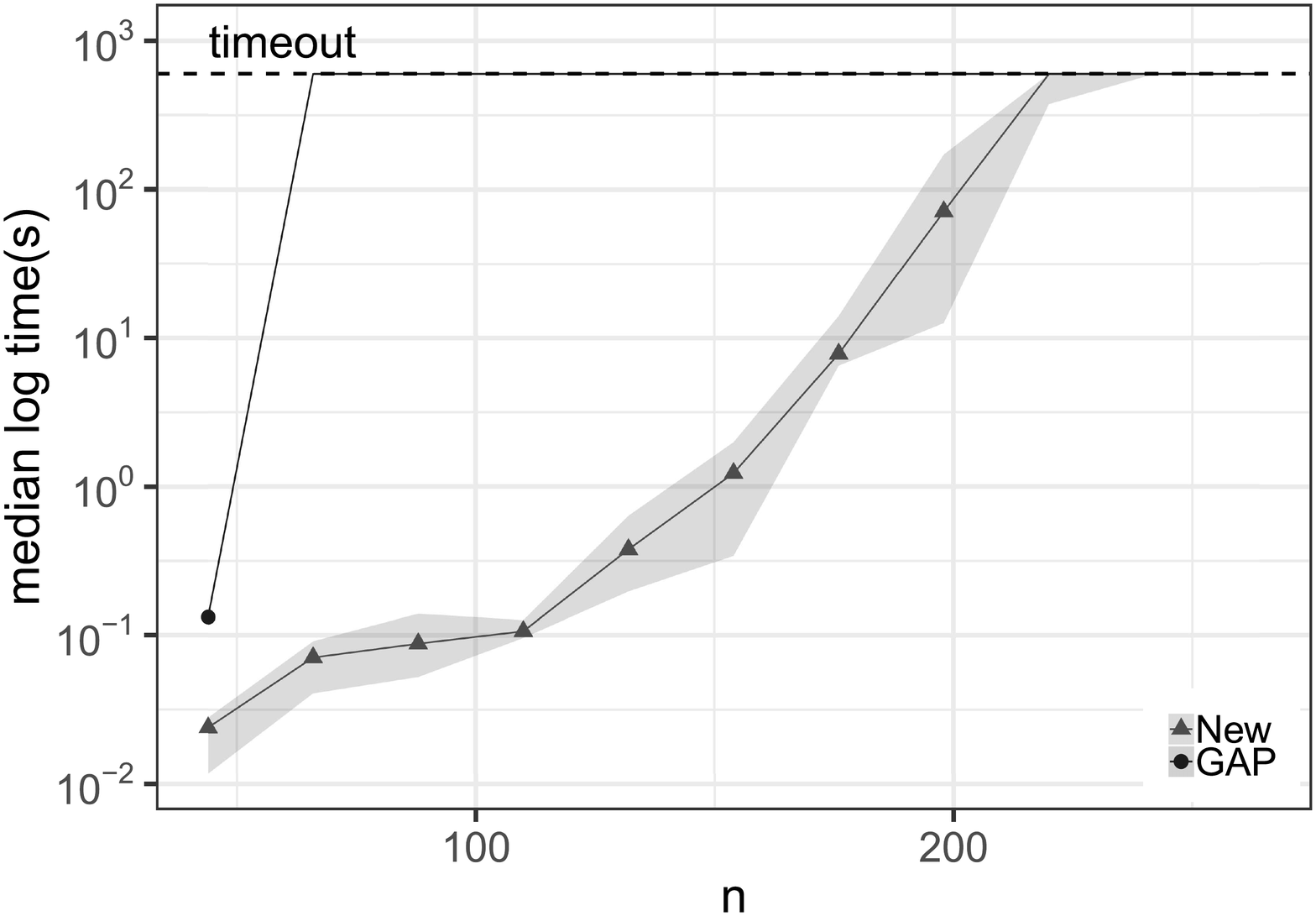}}
   \caption{Median log time (secs) to compute $N_{\Sy_n}(H)$ for 10 instances of $H \leq \Sy_n$ in $\InP(\Cy_p)$. }
   \label{fig:Cpk}
\end{figure}
\begin{figure}[htp!]
   \centering
   \subfloat[$p=3$]{\includegraphics[width=.45\textwidth,height=.25\textheight]{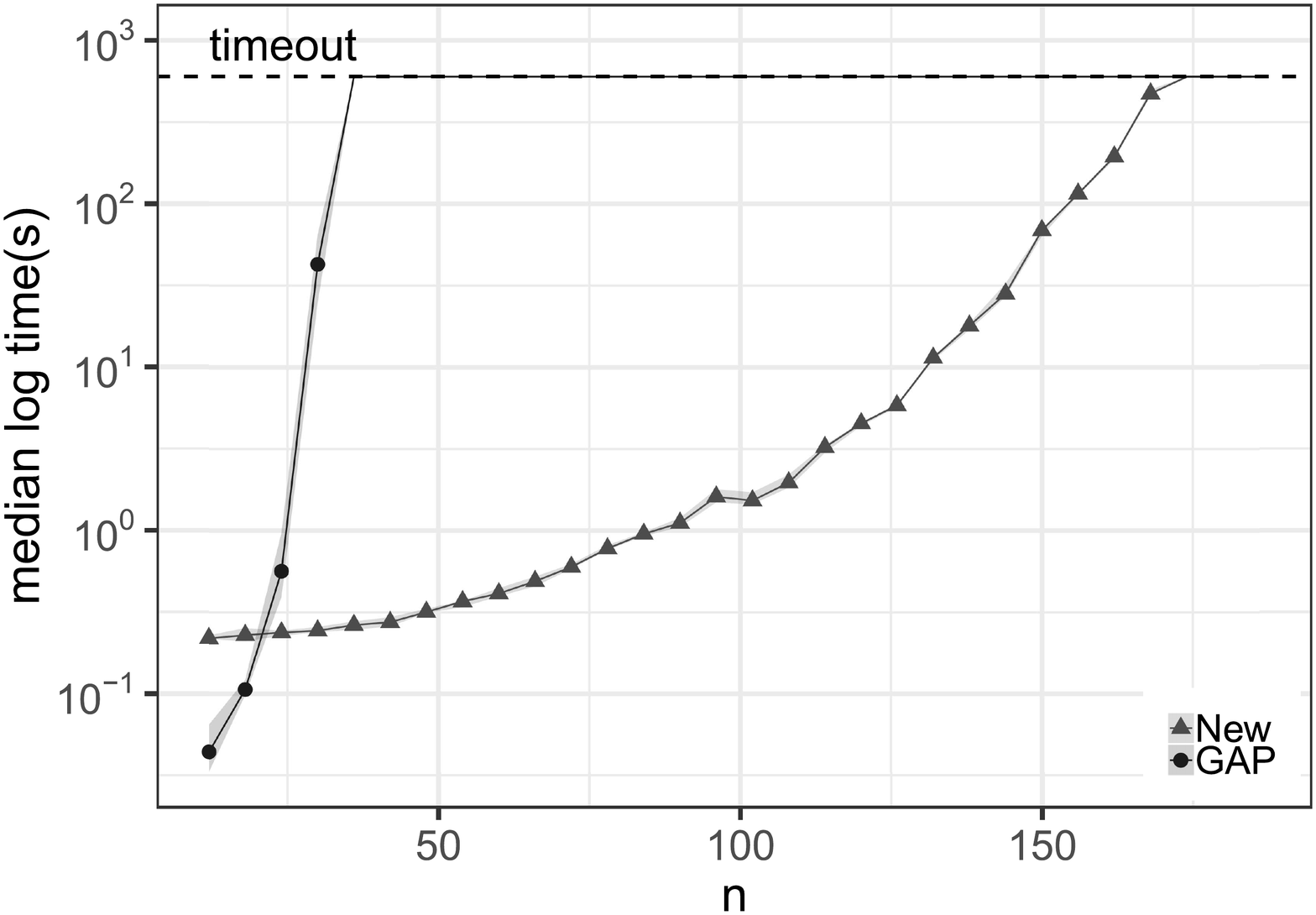}}\quad
   \subfloat[$p=11$]{\includegraphics[width=.45\textwidth,height=.25\textheight]{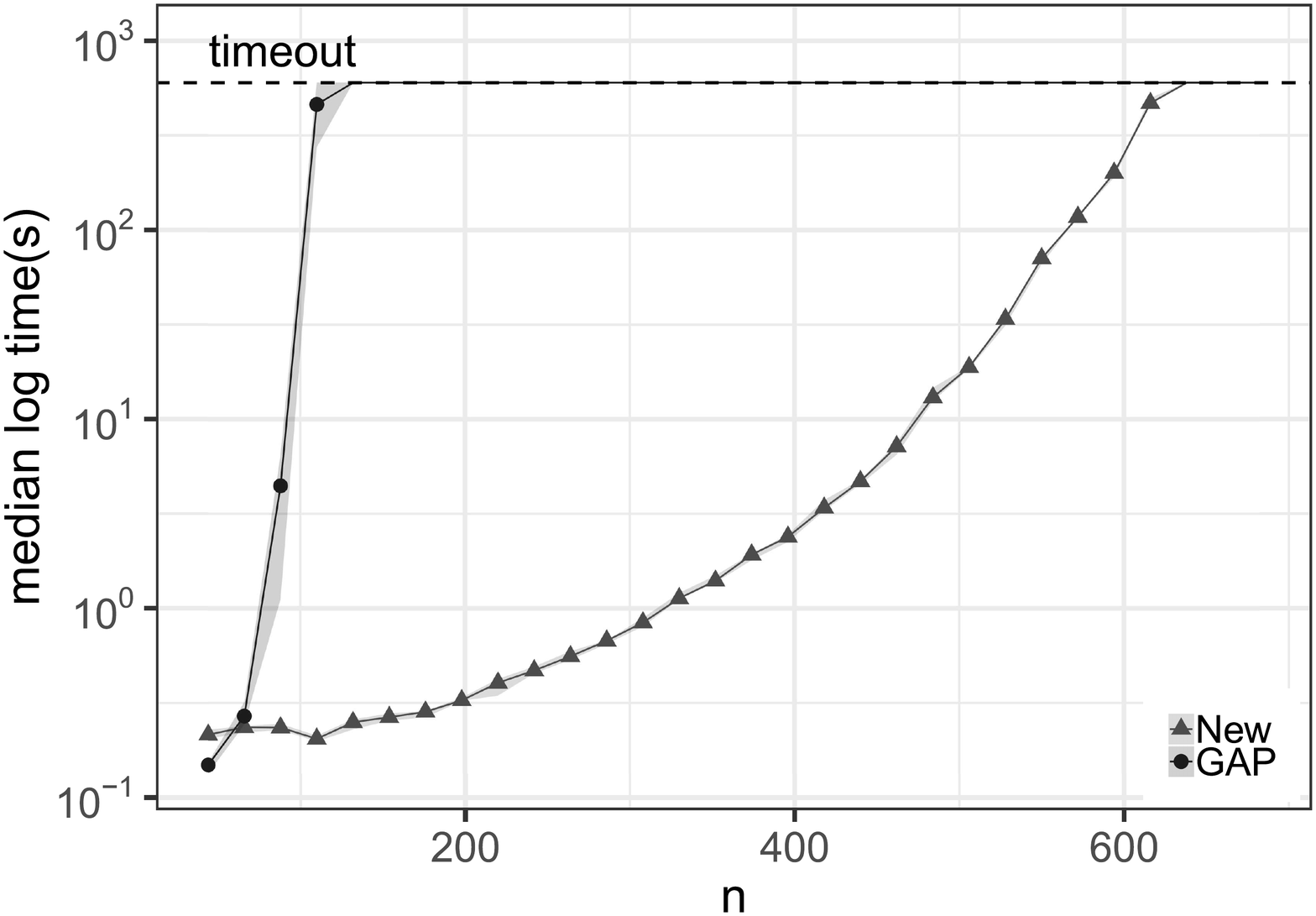}} \\
   \caption{Median log time (secs) to compute $N_{\Sy_n}(H)$ for 10 instances of  $H \leq \Sy_n$ in $\InP(\Di_{2p})$.}
   \label{fig:dpk}
\end{figure}


\section*{Acknowledgements}

The first and third authors would like to thank the Isaac Newton Institute for Mathematical Sciences, Cambridge, for support and hospitality during the programme “Groups, Representations and Applications: New perspectives”, where work on this paper was undertaken. This work was supported by EPSRC grant no EP/R014604/1.
This work was also partially supported by a grant from the Simons Foundation. 
The first and second authors are supported by the Royal Society (RGF\textbackslash EA\textbackslash181005 and URF\textbackslash R\textbackslash180015). 

\newpage

\bibliographystyle{abbrv}

\end{document}